\documentclass[12pt]{article}
\usepackage{amsmath,amssymb,amsfonts,amsthm}

\usepackage{graphicx}
\newenvironment{myabstract}{\par\noindent
{\bf Abstract . } \small }
{\par\vskip8pt minus3pt\rm}
\newcounter{item}[section]
\newcounter{kirshr}
\newcounter{kirsha}
\newcounter{kirshb}
\newenvironment{enumroman}{\setcounter{kirshr}{1}
\begin{list}{(\roman{kirshr})}{\usecounter{kirshr}} }{\end{list}}
\newenvironment{enumarab}{\setcounter{kirshb}{1}
\begin{list}{(\arabic{kirshb})}{\usecounter{kirshb}} }{\end{list}}
\newtheorem{theorem}{Theorem}[section]

\newtheorem{lemma}[theorem]{Lemma}
\newtheorem{corollary}[theorem]{Corollary}

\newenvironment{demo}[1]{\noindent{\bf #1.}\upshape\mdseries}
{\nopagebreak{\hfill\rule{2mm}{2mm}\nopagebreak}\par\normalfont}
\theoremstyle{definition}
\newtheorem{remark}[theorem]{Remark}

\newtheorem{example}[theorem]{Example}
\newtheorem{definition}[theorem]{Definition}

\def\Nr{{\mathfrak{Nr}}}
\def\Fr{{\mathfrak{Fr}}}
\def\Sg{{\mathfrak{Sg}}}
\def\Fm{{\mathfrak{Fm}}}

\def\K{{\mathfrak{K}}}
\def\CA{{\bf CA}}
\def\K{{\bf K}}

\def\Rd{{\ Rd}}
\def\(R)RA{{\bf (R)RA}}

\def\N{\mathbb{N}}

\def\C{\mathbb{C}}
\def\A{{\mathfrak{A}}}
\def\B{{\mathfrak{B}}}
\def\C{{\mathfrak{C}}}
\def\D{{\mathfrak{D}}}
\def\Rd{{\mathfrak{Rd}}}

\def\PA{{\bf PS}}

\def\T{{\bf T}}
\def\T{{\bf T}}

\def\Gp{{\sf Gp}}

\def\PA{{\bf PA}}
\def\T{{\sf T}}

\def\CPEA{{\sf CPEA}}
\def\L{{\mathfrak{L}}}
\def\dom{{\sf dom}}
\def\rng{{\sf rng}}
\def\E{{\mathfrak{E}}}
\def\si{{i_0, i_1, \cdots, i_k}}
\def\sj{{j_0, j_1, \cdots, j_k}}
\def\tr{{[i_0|j_0]|[i_1|j_1]|\cdots |[i_k|j_k]}}
\def\tra{{[i|j_0]|[i_1|j_1]|\cdots |[i_k|j_k]}}

\def\ttr{{(\t^{j_0}_{i_0}\cdots \t^{j_k}_{i_k})}}
\def\ttra{{(\t^{j_0}_{i}\cdots \t^{j_k}_{i_k})}}

\def\ef{Ehren\-feucht--Fra\"\i ss\'e}

\def\tr{{[i_0|j_0]|[i_1|j_1]|\cdots |[i_k|j_k]}}
\def\edges{{\sf edges}}
\def\w{{\sf w}}
\def\g{{\sf g}}
\def\y{{\sf y}}
\def\r{{\sf r}}
\def\t{{\sf t}}

\def\Nr{{\mathfrak{Nr}}}
\def\Fr{{\mathfrak{Fr}}}
\def\Sg{{\mathfrak{Sg}}}
\def\Fm{{\mathfrak{Fm}}}

\def\K{{\mathfrak{K}}}
\def\CA{{\bf CA}}
\def\K{{\bf K}}

\def\Rd{{\ Rd}}
\def\(R)RA{{\bf (R)RA}}

\def\N{\mathbb{N}}

\def\C{\mathbb{C}}
\def\A{{\mathfrak{A}}}
\def\B{{\mathfrak{B}}}
\def\C{{\mathfrak{C}}}
\def\D{{\mathfrak{D}}}

\def\Rd{{\mathfrak{Rd}}}

\def\PA{{\bf PA}}

\def\Ca{{\mathfrak{Ca}}}
\def\d{Dedekind--MacNeille}
\def\nodes{{\sf nodes}}
\usepackage[all]{xy}

\def\s{{\sf s}}

\def\At{{\sf At}}
\def\PA{{\sf PA}}

\def\Cm{{\sf Cm}}
\def\adm{{\sf adm}}
\def\ls { L\"owenheim--Skolem}
\def\Tm{{\sf Tm}}
\def\pa{$\forall$}
\def\pe{$\exists$}
\def\ws{winning strategy}
\def\CA{{\sf CA}}
\def\T{{\sf T}}

\def\Z{{\mathbb{Z}}}
\def\Fm{{\mathfrak{Fm}}}
\def\map{{\sf map}}
\def\c{{\sf c}}
\title{On notions of representability for cylindric--polyadic algebras, and
a solution to the finitizability problem for quantifier logics with equality}
\author{Tarek Sayed Ahmed\\
Department of Mathematics, Faculty of Science,\\
Cairo University, Giza, Egypt.
  }
\begin{document}
\maketitle

\begin{myabstract}
We consider countable so--called rich subsemigroups of $(^\omega\omega,\circ)$; 
each such semigroup $\T$ gives a variety $\sf CPEA_{\T}$ that is axiomatizable by a finite schema of equations taken 
in a countable subsignature
of that of $\omega$--dimensional cylindric--polyadic algebras with equality where substitutions are restricted to maps 
in $\T$.  It is shown that for any such $\sf T$, $\A\in \sf CPEA_{\T}\iff \A$ is representable as a concrete 
set algebra of $\omega$--ary relations. The operations in the signature are set--theoretically interpreted like  in 
polyadic equality set algebras, but such operations are relativized to
a union of cartesian spaces that are not necessarily disjoint. This is a form of guarding semantics.
We show that $\sf CPEA_{\T}$  is canonical and atom--canonical. 
Imposing an extra condition on $\T$, we prove that atomic algebras in $\sf CPEA_\T$ are completely representable and
that $\sf CPEA_\T$ has the super amalgamation property. 
If $\sf T$ is rich and {\it finitely represented}, it is shown that $\CPEA_\T$ is term definitionally equivalent 
to a  finitely axiomatizable Sahlqvist variety. 
Such semigroups exist. This can be regarded as a solution to the central finitizability problem in algebraic logic for first order logic 
{\it with equality} if we do not insist on full fledged commutativity of quantifiers. 
The finite dimensional case is approached from the view point of guarded and clique guarded 
(relativized) semantics of fragments of first order logic using finitely many variables. 
Both positive and negative results are presented.  
\end{myabstract}

\section{Introduction}

{\bf History and overview:}  Polyadic algebras were introduced by Halmos  to provide an algebraic reflection
of the study of first order logic without equality. Later, the algebras were enriched by
diagonal elements to permit the discussion of equality.
That the notion is indeed an adequate reflection of first order logic was
demonstrated by Halmos' representation theorem for locally finite polyadic algebras
(with and without equality). Daigneault and Monk
proved a strong extension of Halmos' theorem, namely, that
every polyadic algebra (without equality) of infinite dimension is representable \cite{DM}.
The proofs of all such results are in essence `Henkin constructions' implemented algebraically
using a neat embedding theorem.
However, this technique no longer works for polyadic algebras {\it with} equality. In this case all algebras have the neat embedding property, but
there are algebras that are not representable \cite{HMT2, Sain}.

Ferenczi \cite{f, Fer} overcame this impasse  by implementing
two successive changes to the theory of Halmos' polyadic equality algebras of infinite dimension $\alpha$ $(\sf PEA_{\alpha})$. First, he changed the signature
by discarding infinitary cylindrifiers (that is cylindrifications on infinite subsets of $\alpha$), but he kept all substitution operators
corresponding to any transformation $\tau:\alpha\to \alpha$. The substitution operator corresponding to $\tau$ is denoted by
${\sf s}_{\tau}$.  If $\A\in \sf PEA_{\alpha}$ and $\tau:\alpha\to \alpha,$
then ${\sf s}_{\tau}$ is a unary operation on $\A$
that is  a Boolean endomorphism.

Next, he weakened the axioms of polyadic equality algebras restricting them to the new strict reduct.
The axiom Ferenczi weakened is that of commutativity of cylindrifiers, so that in the corresponding logic $\exists x\exists y\phi$ is
not always equivalent to $\exists y\exists x\phi$ ($\phi$ a formula). Ferenczi replaced this commutativity axiom by a strictly weaker one.
These significant modifications enabled him
to obtain a strong representability
result  via a neat embedding theorem analogous to the polyadic case (without equality), but using
{\it relativized semantics}.
In this case, every algebra has the neat embedding property (this does not happen for cylindric algebras of dimension $>1$).
Furthermore this property enforces the relativized
representability of the algebra (this does not happen for polyadic {\it equality} algebras).

{\bf The main results:} The theme in {\it relativization} for cylindric--like algebras is (syntactically) weakening the commutativity of
cylindrifiers thus (semantically) moving away from Tarskian
square semantics.  The aim is to diffuse undesirable properties,  like undecidability of the validity problem, and to obtain completeness theorems.
In this paper, we further pursue  this line of research. We show,
using a neat embedding theorem,
that the atomic algebras introduced by Ferenczi, recalled below in definition \ref{capa}, are {\it completely representable}.  An algebra is completely 
representable, if it has a representation 
that carries all meets, possibly infinite, to set--theoretic
intersection.  We also show that the free algebras have a strong interpolation property. 
Most important is that we introduce a countable version of such algebras, and not only do
we prove the countable analogues
of the above two results, but we also prove that
the corresponding infinitary logic with equality
has an  omitting types theorem. This was not possible before because the signature was uncountable, 
and it is well known that omitting types theorems are very much tied to countability via the Baire 
category theorem (though they are usually not presented this way).

Our investigations are in the framework of  what is referred to in the literature as {\it the  semigroup approach in algebraic logic} 
initiated by Craig, and further pursued by Andr\'eka, N\'emeti, Thompson, Sain and others \cite{AT, Nemeti, Sain, SG, Bulletin}. 
The substitution operations ${\sf s}_{\tau}$ in the signature of the variety ${\sf V}_{\T}$ that 
we define and study, are determined by a countable subsemigroup $\T$ of $({}^{\omega}\omega, \circ)$;
we consider only those substitution operations $\s_{\tau}$s for which $\tau\in \T$. 
The signature of $\sf V_\T$ contains, besides the Boolean operations and  $\s_{\tau}$ for all $\tau\in \T$, 
all cylindrifiers and diagonal elements with indices in $\omega$, so it consists of $\omega$--dimensional algebras whose signature expands 
the signature of $\omega$--dimensional cylindric algebras by substitutions indexed by elements of $\T$.

We show that if $\sf T$ is {\it rich} (to be defined below),  then every algebra in ${\sf V}_{\T}$ is representable as a set algebra with top element a set of $\omega$--ary sequences,
and operations interpreted like those of $\omega$--dimensional polyadic equality set algebras restricted
to the signature of ${\sf V}_\T$.  This representability notion (semantics)  does not necessarily respect commutativity of cylindrifiers (quantifiers), 
but it respects a weak form thereof.
We show that $\sf V_\T$  is a Sahlqvist, completely additive conjugated variety,
that is axiomatizable by a recursive finite Halmos' schemata.
Furthermore, ${\sf V}_{\sf T}$ is 
canonical, atom--canonical, and closed under \d\ completions. We also show, that if $\T$ is {\it strongly rich}, a condition stronger than richness as the name suggests, 
then the atomic algebras in $\sf V_\T$ are completely representable, 
and that $\sf V_{\T}$  has the super amalgamation property. 

If $\T$ is rich and {\it finitely presented}, then we show that $\sf V_\T$ is {\it definitionally equivalent} 
to a variety  having a finite signature, and admitting 
a {\it finite equational Sahlqvist axiomatization}. 
Such a semigroup $\T$ 
was constructed by Sain \cite{Sain}. Using such a $\T$,  one can show that 
the finite set $\sf S$ presenting $\T$ defines a finitely axiomatizable  
variety ${\sf V}_{\sf S}$ in the {\it finite signature} expanding the Boolean operations, by only the cylindrifier $\c_0$, 
the diagonal element ${\sf d}_{01}$
and substitution operations ${\sf s}_{\tau}$, $\tau\in \sf S$, such that  
${\sf V}_{\sf S}=\bold I{\sf Gp_\T}$, where  $\sf Gp_{\T}$ denotes the concrete class of algebras (consisting of $\omega$--ary relations) 
representing algebras in ${\sf V}_\T$  and  
 $\bold I$ denotes the operation of taking isomorphic copies. 
In particular, the variety $\bold I\sf Gp_\T$ is, like Boolean set algebras, 
finitely axiomatizable. The corresponding algebraisable logic $\L_\T$ admits {\it a finite,  
sound and complete Hilbert style axiomatization}. For first order logic the {\it Entscheidungsproblem}
posed by Hilbert has a negative answer: The validity problem of first order logic is undecidable. 
The validity problem for $\L_{\T}$ is not settled in this paper. 
Algebraically, we do not know whether the equational theory
of $\bold I{\sf Gp}_{\T}$ is decidable  or not.

We consider our positive (main) results 
a reasonable solution to the finitizability problem for first order logic {\it with} equality \cite{AU, Sain, Bulletin, Nemeti} if we are willing to slightly broaden standard 
Tarskian semantics.
The finitizability problem $(\sf FP)$, seeks a Stone--like representability result for algebras of relations having infinite rank.  
The $\sf FP$, originating with Henkin, Monk and Tarski in the seventies of the last century, 
asks 
for a `nice' variety of representable algebras  whose members are concrete algebras (like Boolean fields of sets and cylindric set algebras) consisting of $\omega$--ary
relations,  where the  operations are set--theoretically defined
(like the Boolean intersection and cylindrifiers interpreted as projections). 
This variety, in addition,  should offer an algebraization (in the standard Blok--Pigozzi sense \cite{bp}) of variants or modifications
of first order logic, and at the same time admits a 
strictly finite equational axiomatization. Dominated by negative results that can be traced back to 
the work of Henkin, Monk and Tarski in the late sixties of the last century \cite{Andreka}, 
this problem has provoked continuous extensive research till the present day. 

The research 
consisted mainly of finding ways to sidestep a long list of non--finite axiomatizability results proved for standard algebraizations 
of $L_{\omega, \omega}$ and its finite variable fragments (as long as the variables available are $>2$), such as (primarily) representable cylindric and quasi--polyadic algebras.
The non--finite axiomatizability results involving dozens of publications, 
were proved by pioneers including Tarski, Andr\'eka, Biro, Johnson, Hirsch, Hodkinson, N\'emeti, Monk, Maddux, Sain, and Thompson. The reader is referred 
to \cite{Nemeti, Bulletin} for an overview.
A satisfactory solution for first order logic without equality, to be recalled below, was provided by Sain \cite{Sain}. 
But for first order logic {\it with equality}, the finitizability problem remained resilient to 
many dedicated  trials.  

We show that our solution is an 
infinite analogue of the finite dimensional algebras studied in \cite{Fer}, in the sense that the class of representable algebras 
in both cases is obtained by relativizing top elements 
to unions of certain spaces (not necessarily disjoint). 
We also show that the universal, hence equational theory, of such finite dimensional varieties 
of representable algebras 
is decidable, so that the validity problem for the corresponding guarded fragment of first order logic is 
decidable. This result is known \cite{ans}, but we provide a new proof using the decidability of the loosely guarded fragment of first order logic.
Throughout the paper, we follow the notation of \cite{1} 
which is in conformity with the notation of the monographs \cite{HMT1, HMT2}. Notation that is possibly unfamiliar  will be 
explained at its first occurrence in the text.

\subsection*{Layout} 

\begin{enumarab}

\item In the following section we prove that atomic cylindric--polyadic equality algebras are completely representable.

\item In section 3, we restrict our investigation to the countable case.

\item Using the results in section 3, in the following section a `non--commutative' solution, moving away, but
only slightly from Tarskian semantics,
is given to the finitizability problem for first order logic with equality.

\item In section 5, we discuss in some depth the status
of the finite dimensional version of the finitizability problem
dealing with guarded and the so--called locally guarded fragments
of first order logic. We prove a new theorem on the failure of the omitting types theorem in a 
strong sense for finite variable 
locally guarded fragments of first order logic, and we prove the aforementioned positive decidability result on finite 
variable guarded fragments of first order logic. 
\end{enumarab}
In the final section our results, together with closely related other (mostly known) results, 
are summarized in tabular form.

\section{Cylindric--polyadic equality algebras}

We start by recalling the abstract equational
axiomatization of algebras considered henceforth.
Their signature is  obtained from that of polyadic equality algebras by discarding infinitary
cylindrifiers. Only finite cylindrifiers remain, so these algebras 
have a cylindric facet, as well; hence their name.
The axiomatization is due to Ferenczi \cite{f}. In this subsection $\alpha$ is an infinite ordinal.

\begin{definition}\label{capa}  By a {\it cylindric--polyadic equality algebra} of dimension $\alpha$,
or a $\CPEA_{\alpha}$ for short,
we understand an algebra of the form
$$\A=(A,+,\cdot ,-,0,1,{\sf c}_{i},{\sf s}_{\tau}, {\sf d}_{ij})_{i, j\in \alpha ,\tau\in {}^{\alpha}\alpha}$$
where ${\sf c}_{i}$ ($i\in \alpha$) and ${\sf s}_{\tau}$ ($\tau\in {}^{\alpha}\alpha)$ are unary
operations on $A$, such that the postulates
below hold for $x,y\in A$, $\tau,\sigma\in {}^{\alpha}\alpha$ and
$i, j\in \alpha$

\begin{enumerate}

\item $(A,+,\cdot, -, 0, 1)$ is a Boolean algebra,
\item ${\sf c}_{i}0=0,$

\item $x\leq {\sf c}_{i}x,$

\item ${\sf c}_{i}(x\cdot {\sf c}_{i}y)={\sf c}_{i}x \cdot {\sf c}_{i}y,$

\item ${\sf s}_{\tau}$ is a Boolean endomorphism,

\item ${\sf s}_{Id}x=x,$

\item ${\sf s}_{\sigma\circ \tau}={\sf s}_{\sigma}\circ {\sf s}_{\tau},$

\item  ${\sf d}\cdot {\sf s}_{\sigma}x={\sf d}\cdot {\sf s}_{\tau}x$ if the product $\sf d$
of the elements ${\sf d}_{\tau(i), \sigma(i)} (i\in \Delta x)$ exists,

\item ${\sf c}_i{\sf s}_{\sigma}x\leq {\sf s}_{\sigma}{\sf c}_jx$
if $\sigma^{-1}\{i\}$ equals $\{j\}$ or is the empty set,
and equality holds in place of  $\leq$ if $\sigma$ is a permutation,

\item ${\sf d}_{ii}=1,$

\item $x\cdot {\sf d}_{ij}\leq {\sf s}_{[i|j]}x,$

\item ${\sf s}_{\tau}{\sf d}_{ij}={\sf d}_{\tau(i), \tau(j)}$.

\end{enumerate}

\end{definition}

The axiom in item (9) 
is substantially weaker than that of commutativity of cylindrifiers.
Let $\Gp_{\alpha}$ be the class of representable algebras \cite[Definition 6.3.2]{f}. The top element of such algebras is
a union of {\it cartesian spaces} that are not necessarily
disjoint (as is the case with cylindric algebras)
and the operations are interpreted in the usual concrete sense, like polyadic equality algebras
relativizing the available operations to top elements. 

A {\it cartesian space} is a set of the form $^{\alpha}U$ for some non--empty set $U$.
It is tedious but routine to check that all axioms hold in such algebras. This is a soundness theorem.
Conversely, Ferenczi proved {\it completeness}, namely, $\CPEA_{\alpha}\subseteq \sf Gp_{\alpha}$ \cite{f}.
Next we show that any such algebra, when atomic,
admits a {\it complete relativized representation} in the following sense:

\begin{definition} Let $\A\in \sf CPEA_{\alpha}$.
Then $\A$ is {\it completely representable} if there exist $\B\in \Gp_{\alpha}$ and an isomorphism
$f:\A\to \B$ such that for all $X\subseteq \A$,
$f(\prod X)= \bigcap f(X)$ whenever $\prod^{\A}X$ exists.
\end{definition}

We say that $f:\A\to \B$ is a {\it complete representation} of $\A$.
It is known \cite{HH} that $f:\A\to \B$ is a complete representation of $\A$ $\iff$ $\A$ is atomic and completely additive 
and $f$ is {\it atomic},  in the sense that $\bigcup_{x\in \At\A}f(x)=1^{\B}$.

The proof of the following
theorem is similar to the proof of the main result in \cite{Sayedpa}.
Before embarking on the proof, we need the following
crucial definitions. We write $Id_X$ for the identity function on $X$.
Sometimes we write only $Id$ if $X$ is clear from the context.
\begin{definition} 
\begin{enumarab}
\item Let  $\alpha<\beta$ be infinite ordinals and $\B\in \CPEA_{\beta}$. Then the {\it $\alpha$--neat reduct} of $\B$, in symbols
$\Nr_{\alpha}\B$, is the
algebra obtained from $\B$, by discarding
cylindrifiers and diagonal elements whose indices are in $\beta\sim \alpha$, and restricting the universe to
the set $Nr_{\alpha}B=\{x\in \B: \{i\in \beta: {\sf c}_ix\neq x\}\subseteq \alpha\}.$
For $\tau\in {}^{\alpha}\alpha$ the substitution operator ${\sf s}_{\tau}$
is defined by  ${\sf s}_{\bar{\tau}}^{\B}$, where $\bar{\tau}=\tau\cup Id_{\beta\sim \alpha}$.

\item  A {\it transformation system} is a quadruple of the form $(\A, I, G, {\sf S})$ where $\A$ is an 
algebra of any signature, 
$I$ is a non--empty set (we will only be concerned with infinite sets),
$G$ is a subsemigroup of $(^II,\circ)$ (the operation $\circ$ denotes composition of maps) 
and ${\sf S}$ is a homomorphism from $G$ to the semigroup of endomorphisms of $\A$. 
Elements of $G$ are called transformations. 
\end{enumarab}
\end{definition}
In the following proof we use that our algebras are completely additive.  
The next theorem implies the representability result of Ferenczi \cite{f}, because $\CPEA_{\alpha}$ is Sahlqvist axiomatizable,
so it is {\it canonical}. Given $\A\in \CPEA_{\alpha}$, then $\A$ embeds into its completey representable atomic 
canonical extension, so it
will be representable. The theorem 
also has an interesting metalogical interpretation. 
The corresponding logic which is a non--commutative fragment of Keisler's logic \cite{K} 
has  a `Vaught theorem':  Atomic theories have atomic models. 
Witness \cite{conference, Sayedpa} for an analogous situation
for  other fragments of Keisler's 
logic including itself.

\begin{theorem}\label{main} Every atomic $\CPEA_{\alpha}$ is completely representable. In particular, the class of completely representable 
$\sf CPEA_{\alpha}$s is elementary.
\end{theorem}
\begin{proof} Let $\A\in \CPEA_{\alpha}$ be atomic. Let $c\in \A$ be non--zero. We will find a $\C\in \Gp_{\alpha}$ and a homomorphism
$f:\A \to \C$ that preserves arbitrary suprema
whenever they exist and also satisfies that  $f(c)\neq 0$.
This homomorphism may not be injective.
Let ${\sf End}(\A)$ be the semigroup of Boolean endomorphisms on
$\A$.  Then the map $\sf S:{}^\alpha\alpha\to {\sf End}(\A)$ defined  via $\tau\mapsto {\sf s}_{\tau}$ is a homomorphism of semigroups.
The operation on both semigroups is composition of maps, so that $(\A, \alpha, {}^{\alpha}\alpha, \sf S)$ is a transformation system. 
For any set $X$, let $F(^{\alpha}X,\A)$
be the set of all maps from $^{\alpha}X$ to $\A$ endowed with Boolean  operations defined pointwise and for
$\tau\in {}^\alpha\alpha$ and $f\in F(^{\alpha}X, \A)$, put ${\sf s}_{\tau}f(x)=f(x\circ \tau)$.

This turns $F(^{\alpha}X,\A)$ to a transformation system as well.
The map $H:\A\to F(^{\alpha}\alpha, \A)$ defined by $H(p)(x)={\sf s}_xp$ is
easily checked to be an embedding of transfomation systems. Assume that $\beta\supseteq \alpha$. Then $K:F(^{\alpha}\alpha, \A)\to F(^{\beta}\alpha, \A)$
defined by $K(f)x=f(x\upharpoonright \alpha)$ is an embedding, too.
These facts are fairly straightforward to establish
\cite[Theorems 3.1, 3.2]{DM}.

Call $F(^{\beta}\alpha, \A)$ a minimal functional dilation of $F(^{\alpha}\alpha, \A)$.
Elements of the big algebra, or the (cylindrifier free)
functional dilation, are of form ${\sf s}_{\sigma}p$,
$p\in F(^{\beta}\alpha, \A)$ where $\sigma\upharpoonright \alpha$ is injective \cite[Theorems 4.3-4.4]{DM}.

Let $\B$ be the algebra obtained from $\A$, by discarding its cylindrifiers, then taking  a minimal functional dilation,
dilating $\A$ to a regular cardinal $\mathfrak{n}$.\footnote{If $\kappa$ is a cardinal, then {\it the cofinality of $\kappa$},
in symbols ${\sf cf}\kappa$, is  the least cardinal $\lambda$ such that $\kappa$ is the union of $\lambda$ sets each having cardinality $<\kappa$. The cardinal
$\kappa$ is {\it regular} if ${\sf cf}\kappa=\kappa$.} We also require that $|\mathfrak{n}|>|\alpha|$
and $|\mathfrak{n}\sim \alpha|=\mathfrak{n}$.
One re-defines cylindrifiers in the dilation $\B$ by setting for each  $i\in  \mathfrak{n}:$
$${\sf c}_{i}{\sf s}_{\sigma}^{\B}p={\sf s}_{\rho^{-1}}^{\B} {\sf c}_{(\rho(i)\cap \sigma(\alpha))}^{\A}
{\sf s}_{(\rho\sigma\upharpoonright \alpha)}^{\A}p.$$
Here $\rho$ is any permutation such that $\rho\circ \sigma(\alpha)\subseteq \sigma(\alpha).$
The definition is sound, that is, it is independent of $\rho, \sigma, p$; furthermore, it agrees with the old cylindrifiers in $\A$.
Identifying algebras with their transformation systems
we get that $\A\cong \Nr_{\alpha}\B$, via the isomorphisn $H$ defined
for $f\in \A$ and $x\in {}^{\mathfrak{n}}\alpha$ by,
$H(f)x=f(y)$ where $y\in {}^{\alpha}\alpha$ and $x\upharpoonright \alpha=y$,
\cite[Theorem 3.10]{DM}.
This dilation also has Boolean reduct isomorphic to $F({}^\mathfrak{n}\alpha, \A)$, in particular, it is atomic because $\A$ is atomic
(a product of atomic Boolean algebras is atomic). For $\tau\in {}^\mathfrak{n}\mathfrak{n}$,
$\dom\tau=\{i\in \mathfrak{n}: \tau(i)\neq i\}$ and $\rng(\tau)=\{\tau(i): i\neq \tau(i)\}$.
Let $\sf adm$ be  the set of {\it admissible substitutions}. The transformation $\tau\in {}^{\mathfrak{n}}\mathfrak{n}$
is {\it admissible} if
$\dom\tau\subseteq \alpha$ and $\rng\tau\cap \alpha=\emptyset$, so that $\rng\tau\subseteq \mathfrak{n}\sim \alpha$.
Then we have
for all $j< \mathfrak{n}$, $p\in \B$ and $\sigma\in \sf adm$,
\begin{equation}\label{tarek1}
\begin{split}
{\sf s}_{\sigma}{\sf c}_{j}p=\sum_{i<\mathfrak{n}} {\sf s}_{\sigma}{\sf s}_i^jp
\end{split}
\end{equation}
The last supremum uses that ${\sf c}_kp=\sum_{i<\mathfrak{n}} {\sf s}_i^k p$, which is proved like the cylindric case \cite[Theorem 1.11.6]{HMT1}.
Let $X$ be the set of atoms of $\A$. Since $\A$ is atomic, then  $\sum^{\A} X=1$. By $\A=\Nr_{\alpha}\B$ we also have $\sum^{\B}X=1$
because $\A$ is a complete subalgebra of $\B$, that is if $S\subseteq \A$ and $y\in A$
is such that $\sum ^{\A}S=y$, then $\sum^{\B}S=y$.

To see why assume that $S\subseteq \A$ and $\sum ^{\A}S=y$, and for contradiction that there exists $d\in \B$ such that
$s\leq d< y$ for all $s\in S$. Then $d$ uses finitely many dimensions not in $\alpha$, say $m_1,\ldots, m_n$.
Let $t=y\cdot -{\sf c}_{m_1}\ldots {\sf c}_{m_n}(-d)$  (here the order of cylindrifiers makes a difference
because cylindrifiers do not commute but in this context the order is immaterial,
any fixed order will do). We claim that $t\in \A=\Nr_{\alpha}\B$ and $s\leq t<y$ for all $s\in S$. This contradicts
$y=\sum^{\A}S$.
The first required follows from the fact that $\Delta y\subseteq \alpha$ and that all indices in $\mathfrak{n}\sim \alpha$ that occur in $d$
are cylindrified.  In more detail, put $J=\{m_1, \ldots, m_n\}$ (such that cylindrification on $J$ is taken in this order) and let 
$i\in \mathfrak{n}\sim \alpha$, then:
$${\sf c}_{i}t={\sf c}_{i}(-{\sf c}_{(J)} (-d))=
 {\sf c}_{i}-{\sf c}_{(J)} (-d)$$
$$={\sf c}_{i} -{\sf c}_{i}{\sf c}_{(J)}( -d)=
-{\sf c}_{i}{\sf c}_{(J)}( -d)=
-{\sf c}_{(J)}( -d)=t.$$
We have shown that ${\sf c}_it=t$ for all $i\in \mathfrak{n}\sim \alpha$, hrene $t\in \Nr_{\alpha}\B=\A$. If $s\in S$, we show that $s\leq t$. We know that $s\leq y$. Also $s\leq d$, so $s\cdot -d=0$.
Hence $0={\sf c}_{m_1}\ldots {\sf c}_{m_n}(s\cdot -d)=s\cdot {\sf c}_{m_1}\ldots {\sf c}_{m_n}(-d)$, so
$s\leq -{\sf c}_{m_1}\ldots {\sf c}_{m_n}( -d)$, hence  $s\leq t$ as required. We finally check that $t<y$. If not, then
$t=y$ so $y \leq -{\sf c}_{m_1}\ldots {\sf c}_{m_n}(-d)$ and so $y\cdot  {\sf c}_m\ldots {\sf c}_{m_n}(-d)=0$.
But $-d\leq {\sf c}_m\ldots {\sf c}_{m_n}(-d)$,  hence $y\cdot -d\leq y\cdot  {\sf c}_m\ldots {\sf c}_{m_n}(-d)=0.$
Hence $y\cdot -d =0$ and this contradicts that $d<y$. We have proved that $\sum^{\B}X=1.$

Because substitutions are completely additive, we get:
\begin{equation}\label{tarek2}
\begin{split}
(\forall \tau\in {\sf adm})({\sf s}_{\tau}^{\B}X=1).
\end{split}
\end{equation}
Let $S$ be the Stone space of $\B$, whose underlying set consists of all Boolean ultrafilters of
$\B$. Let $X^*$ be the set of principal ultrafilters of $\B$ (those generated by the atoms).
These are isolated points in the Stone topology, and they form a dense set in the Stone topology since $\B$ is atomic.
So we have $X^*\cap T=\emptyset$ for every nowhere dense set $T$.
For $a\in \B$, let $N_a$ denote the set of all Boolean ultrafilters containing $a$.
Now  for all $i\in \alpha$, $p\in \B$ and $\tau\in {\sf adm}$ we have,
by the suprema, evaluated in (1) and (2):
\begin{equation}\label{tarek3}
\begin{split}
G_{\tau, i,p}=N_{{\sf s}_{\tau}{\sf c}_{i}p}\sim \bigcup_{j\in \mathfrak{n}, \tau\in \sf adm} N_{{\sf s}_{\tau}{\sf s}_j^ip}
\end{split}
\end{equation}
and
\begin{equation}\label{tarek4}
\begin{split}
G_{X, \tau}=S\sim \bigcup_{x\in X}N_{s_{\tau}x}.
\end{split}
\end{equation}
are nowhere dense in the Stone topology $S$.
Take $F$ to be any principal ultrafilter of $S$ containing $c$.
This is possible since $\B$ is atomic, so there is an atom $x$ below $c$; just take the
ultrafilter generated by $x$.
Then $F\in X^*$, so $F\notin G_{\tau, i, p}$, $F\notin G_{X,\tau},$
for every $i\in \alpha$, $p\in B$
and $\tau\in \sf adm$. By condition (4) and definition,   $F$ is a {\it perfect ultrafilter} \cite[pp.128]{Sayedneat}.

Let $\Gamma=\{i\in \mathfrak{n}:\exists j\in \alpha: {\sf c}_i{\sf d}_{ij}\in F\}$.
Since $\c_i{\sf d}_{ij}=1$, then $\alpha\subseteq \Gamma$. Furthermore the inclusion is proper, 
because for every $i\in \alpha$, there is a $j\not\in \alpha$ such that ${\sf d}_{ij}\in F$. 
Define the relation $\sim$ on $\Gamma$ via  $m\sim n\iff {\sf d}_{mn}\in F.$ 
Then $\sim$ is an equivalence relation because for all $i, j, k\in \alpha$, ${\sf d}_{ii}=1\in F$, ${\sf d}_{ij}={\sf d}_{ji}$,
${\sf d}_{ik}\cdot {\sf d}_{kj}\leq {\sf d}_{lk}$ and filters are closed upwards.
Now we show that the required  representation will be a $\Gp_{\alpha}$ with base
$M=\Gamma/\sim$. One defines the  homomorphism $f$ like in \cite[pp.128-129]{Sayedneat}
using the hitherto obtained perfect ultrafilter $F$ as follows:
For $\tau\in {}^{\alpha}\Gamma$, 
such that $\rng(\tau)\subseteq \Gamma\sim \alpha$ (the last set is non--empty, because $\alpha\subsetneq \Gamma$),
let  $\bar{\tau}: \alpha\to M$ be defined by $\bar{\tau}(i)=\tau(i)/\sim$  
and write $\tau^+$ for $\tau\cup Id_{\mathfrak{n}\sim \alpha}$. 
Then $\tau^+\in \sf adm$, because $\tau^+\upharpoonright \alpha=\tau$, $\rng(\tau)\cap \alpha=\emptyset,$ 
and $\tau^+(i)=i$ for all $i\in \mathfrak{n}\sim \alpha$. 

Let $V=\{\bar{\tau}\in {}^{\alpha}M: \tau: \alpha\to \Gamma,  \rng(\tau)\cap \alpha=\emptyset\}.$
Then $V\subseteq {}^{\alpha}M$ is non--empty (because $\alpha\subsetneq \Gamma$).
Now define $f$ with domain $\A$ via: 
$a\mapsto \{\bar{\tau}\in V: \s_{\tau^+}^{\B}a\in F\}.$ 
Then $f$ is well defined, that is, whenever $\sigma, \tau\in {}^{\alpha}\Gamma$ and 
$\tau(i)\sim \sigma(i)$ for all $i\in \alpha$, then for any $a\in \A$, 
$\s_{\tau^+}^{\B}a\in F\iff \s_{\sigma^+}^{\B}a\in F$.
The congruence relation just defined
on $\Gamma$  guarantees
that the hitherto defined homomorphism respects the diagonal elements.
For the other operations,
preservation of cylindrifiers
is guaranteed by the condition that $F\notin G_{\tau, i, p}$ for all $\tau\in {\sf adm}, i\in \alpha$
and all $p\in A$.

Moreover $f$ is an
atomic representation since by (3) $F\notin G_{X,\tau}$ for every $\tau\in \sf adm$,
which means that for every $\tau\in  \sf adm$
there exists $x\in X$, such that
${\sf s}_{\tau}^{\B}x\in F$, and so $\bigcup_{x\in X}f(x)=V.$
We conclude that $f$ is a complete  representation, since it is an atomic one.
To obtain  $\C\in \Gp_{\alpha}$ and a complete representation
from $\A$ to $\C$, one takes the subdirect product of set algebras constructed for each non-zero element of
$\A$.
\end{proof}

\section{The countable case}

Now we address a {\it countable version} of cylindric--polyadic equality algebras.
For a start, we define certain cardinals that will play a key role in some omitting types theorems that we will prove in a while.
\begin{itemize}
\item Let $\mathfrak{p}$ be the least cardinal  $\kappa$ such that there are $\kappa$ many meager sets  of $\mathbb{R}$ whose union is not meager.
If $\lambda<\mathfrak{p}$, and $(A_i: i<\lambda)$ is a family of meager subsets of a Polish space $X$\footnote{A {\it Polish space} is a topological space that is metrizable with a complete separable metric; 
the real line and the Cantor set are the prime examples.},
then $\bigcup_{i\in \lambda}A_i$ is meager.
The cardinal $\sf cov K$ is the least cardinal such the  Baire category theorem
for Polish spaces fails. If $X$ is a Polish space, then it cannot be covered by $<\sf covK$ many meager sets.
\item The cardinals $\sf cov K$ and $\mathfrak{p}$ are uncountable cardinals, such that
$\mathfrak{p}\leq \sf cov K\leq 2^{\omega}$. It is consistent that $\mathfrak{p}< \sf covK.$
\end{itemize}
For the definition and required properties of $\mathfrak{p}$, witness \cite[pp.3, pp.44-45, Corollary 22c]{Fre}. 
For properties of $\sf cov K,$
witness \cite[The remark on pp.217]{Sayed}.

Since any second countable compact Hausdorff space is Polish, the above properties apply to Stone spaces of countable Boolean algebras.
We specify the new countable signature. The substitution operations will come from a certain countable semigroup.
Since everything is countable, we fix the dimension to be the least infinite ordinal, namely,  $\omega$.
But we shall deal with algebras having dimension $\alpha$, $\alpha$ a countable ordinal, 
mostly $\alpha$  will be $\omega+n$ with $n\leq \omega$.

We will use the semigroup as a superscript in place of the countable dimension $\alpha$,
that is, we write
${\sf CPEA_{\T}}$, for ${\sf CPEA}_{\alpha}$, where $\sf T$ is the subsemigroup
of $({}^\alpha\alpha, \circ)$ 
specifying the signature. We say simply that $\sf T$ is a semigroup on $\alpha$.
By the same token, set algebras are denoted by
${\sf Gp}_\T$.  The dimension will be implicit in $\T$.
To define the countable
semigroups that specify the signature,
of algebras to be addressed, we need some preparation to do.

The definition of rich and strongly rich semigroups to be formulated next is exactly like in \cite[Definition 1.4]{AU}
to which we refer for notation used.

\begin{definition}\label{rich} Let $\alpha$ be a countable ordinal.
Let $\T$ be a subsemigroup of $({}^{\alpha}\alpha, \circ)$.
We say that $\T$ is {\it rich } if $\T$ satisfies the following conditions:
\begin{enumerate}
\item $(\forall i,j\in \alpha)(\forall \tau\in \T) \tau[i|j]\in \T.$
\item There exist $\sigma,\pi\in \T$, called {\it distinguished elements} of $\T$, such that
$(\pi\circ \sigma=Id,\  \rng\sigma\neq \alpha), $ satisfying
$$ (\forall \tau\in \T)(\sigma\circ \tau\circ \pi)[(\alpha\sim \rng\sigma)|Id]\in \T.$$
\end{enumerate}
\end{definition}
\begin{definition}\cite[Definition 1.4]{AU}.\label{stronglyrich}
Let $\T$ be rich a subsemigroup of $({}^{\alpha}\alpha, \circ)$.
Let $\sigma$ and $\pi$ be as in the previous definition.
If $\sigma$ and $\pi$ satisfy:
\begin{enumerate}
\item $(\forall n\in \omega) |\sf sup(\sigma^n\circ \pi^n)|<\omega, $
\item $(\forall n\in \omega)[\sf sup(\sigma^n\circ \pi^n)\subseteq
\alpha\smallsetminus \rng(\sigma^n)];$
\end{enumerate}
then we say that $\T$ is  {\it a strongly rich} semigroup.
\end{definition}

\begin{example}\label{semigroup}
\begin{enumarab}
\item The semigroup $\sf T$ generated by the set of 
transformations 
$\{[i|j], [i,j], i, j\in \omega, \sf suc, \sf pred\}$ defined on $\omega$ is a 
strongly rich subsemigroup of $(^\omega\omega, \circ)$.
Here $\sf suc$ abbreviates the {\it successor function}
on $\omega$, ${\sf suc}(n)=n+1$, and $\sf \sf pred$ acts as its quasi--right inverse, the 
{\it predecessor
function} on $\omega$, defined by
${\sf pred}(0)=0$ and for other $n\in \omega$, ${\sf pred}(n)=n-1.$

\item \cite{AU, Sain}. One can take $\mathbb{Z}$ as an indexing set instead of $\omega$. 
We denote a function $f$ by $(f(x): x\in \dom f)$.
Let ${\sf T}\subseteq ({}^\mathbb{Z}\mathbb{Z}, \circ)$ 
be the semigroup generated by the following five transformations: 
${\sf shift}=(z+1: z\in \mathbb{Z})$ which is a bijection, 
${\sf shift}^{-1}$, ${\sf suc}=(n+1: n\leq 0)\cup (n: n>0)$ 
and ${\sf pred} =(n: n\leq 0)\cup (n-1: n\geq 1)$ 
and the transposition $[0, 1]$ interchanging $0$ and $1$.
\end{enumarab}
In both cases, the transformation $\sf suc$
plays the role of $\sigma$ while the transformation $\sf pred$ plays the role of
$\pi$, hence $\sf suc$ and $\sf pred$
are the distinguished elements of $\sf T$.
\end{example}

The axiomatization of our algebras is exactly  the same
as the axiomatization in definition \ref{capa}
by restricting the previous signature to the new countable signature.

Our next theorem is crucial.
It says that  rich semigroups are adequate to form $\omega$--dilations for countable algebras, so that algebras in $\sf CPEA_T$ have
the neat embedding property.  The definition of neat reducts and dilations is exactly like the ${\sf CPEA}_{\alpha}$ case
by implementing the  obvious modifications. For an algebra $\A$, and $X\subseteq A$, $\Sg^{\A}X$ denotes the subalgebra of $\A$ generated by $X$. 

\begin{lemma}\label{dl} Let $\T$ be
a countable rich subsemigroup of $({}^\omega\omega, \circ)$
and $\A\in \sf CPEA_{\sf T}$ be countable.   Then there exist a  rich semigroup $\sf S$ on $\omega+\omega$ and an $\omega$--dilation 
$\B\in \CPEA_{\sf S}$ of $\A$, so that $\A\subseteq \Nr_{\omega}\B$. If in addition $\T$ is strongly rich, then $\sf S$ can be chosen to be strongly rich, and in this case, 
for all $X\subseteq A,$  $\Sg^{\A}X=\Nr_{\omega}\Sg^{\B}X$.
In particular, $\A=\Nr_{\omega}\B$ and for all $x\in B$, $|\Delta x \sim \omega|<\omega$.
\end{lemma}
\begin{proof}
We assume a particular rich semigroup $\T$, namely, that generated by finite transformations together with
$\sf suc$, $\sf pred$, together with all replacements and transpositions.\footnote{We note that transpositions are definable \cite{AU}.}
The general case is entirely analogous \cite[Remark 2.8, pp.327]{AU}.
We follow verbatim \cite[pp.323--336]{AU}, except that, in addition, we have to check that homomorphisms
hitherto defined preserve the diagonal elements.
This part easily follows from the axiom that ${\sf s}_{\tau}{\sf d}_{ij}={\sf d}_{\tau(i),\tau(i)}.$

For $n\leq \omega$, let  $\alpha_n=\omega+n$
and $M_n=\alpha_n\sim \omega$.
Note that when $n\in \omega$, then $M_n=\{\omega,\ldots,\omega+n-1\}$.
For $\tau\in \T$, let $\tau_n=\tau\cup Id_{M_n}$. $\T_n$ denotes the
subsemigroup of $({}^{\alpha_n}\alpha_n,\circ )$ generated by
$\{\tau_n:\tau\in G\} \cup \{[i|j],[i,j]: i< j\in \alpha_n\}$.
For $n\in \omega$,  let $\rho_n:\alpha_n\to \omega$
be the bijection defined by
$\rho_n\upharpoonright \omega={\sf suc}^n$ and $\rho_n(\omega+i)=i$ for all $i<n$.
Let $n\in \omega$. For $v\in T_n,$ let $v'=\rho_n\circ v\circ \rho_n^{-1}$.
Then $v'\in G$.
For $\tau\in \T_{\omega}$, let
$D_{\tau}=\{m\in M_{\omega}:\tau^{-1}(m)=\{m\}=\{\tau(m)\}\}$.
Then $|M_{\omega}\sim D_{\tau}|<\omega.$

Let $\A$ be the given  countable algebra in $\sf CPEA_T$.
Let $\A_n$ be the algebra defined as follows:
$\A_n=( A, +, \cdot, 0, 1, {\sf c}_i^{\A_n},{\sf s}_v^{\A_n}, {\sf d}_{ij})_{ij\in \alpha_n,v\in \T_n}$
where for each $i\in \alpha_n$ and $v\in \T_n$,
${\sf c}_i^{\A_n}:= {\sf c}_{\rho_n(i)}^{\A} \text { and }{\sf s}_v^{\A_n}:= {\sf s}_{v'}^{\A}.$
Let $\Rd_{\omega}\A_n$ be the following reduct of
$\A_n$ obtained by restricting the signature of $\A_n$ to the first
$\omega$ dimensions:
$\Rd_{\omega}\A_n=(A_n, +, \cdot, -, 0, 1, {\sf c}_i^{\A_n}, {\sf s}_{\tau_n}^{\A_n}, {\sf d}_{ij})_{i,j\in \omega,\tau\in \T}.$
For $x\in A$, let $e_n(x)={\sf s}_{suc^n}^{\A}(x)$. Then $e_n:A\to A_n$ and  $e_n$ is an embedding from $\A$ into $\Nr_{\omega}\A_n$.
From strong richness of $\T$, it follows that $e_n(\Sg^{\A}Y)=\Nr_{\omega}(\Sg^{\A_n}e_n(Y))$ for all
$Y\subseteq A$, cf. \cite[Claim 2.7]{AU}.
While $\sigma$ and condition (2) in definition  \ref{rich},
are needed to implement the neat embedding, the left inverse $\pi$ of $\sigma$, together with the condition of {\it strong richness} is needed to
show that forming neat reducts commute with forming subalgebras, in the sense that (upon identifying $e_n$ with the identity map) for all 
$X\subseteq \A$, $\Sg^{\A}X=\Sg^{\Nr_{\omega}\A_n}X=\Nr_{\omega}\Sg^{\A_n}X$.
In particular, $\A$ is the full $\omega$--neat reduct
of $\A_n$, that is, $\A=\Nr_{\omega}\A_n$.  

Now let $\alpha=\omega+\omega$. To extend the neat embedding part to infinite dimensions, one constructs
 $\B\in \sf CPEA_{\sf S}$ as an ultraproduct of expansions $\A_n^+$ ($n\in \omega$) of the algebras $\A_n$ 
to the signature of ${\sf CPEA}_{\sf S}$, relative to any non--trivial  ultrafilter $U$ say, on $\omega$ \cite{AU}.
Here $\sf S$ is the subsemigroup of $({}^\alpha\alpha, \circ)$ generated by the set $\{\bar{\tau}: \tau\in \T\}\cup \{[i,j], [i|j]: i< j\in \alpha\}$, 
where $\bar{\tau}=\tau\cup Id_{\alpha\sim \omega}$. Then $\A$ neatly embeds into $\B=\Pi_{n/U}\A_n^+$, that is $\A\subseteq \Nr_{\omega}\B$ \cite{AU, Sain}.
Using strong richness of $\T$,  one proves, exactly like in \cite{AU}, that $\sf S\subseteq ({}^{\alpha}\alpha, \circ)$ is strongly rich, too, and  
that for all $X\subseteq \A$, $\Sg^{\A}X=\Sg^{\Nr_{\omega}\B}X=\Nr_{\omega}\Sg^{\B}X$.
\end{proof}
We need some more definitions.
\begin{definition}\label{o}
\begin{enumarab}
\item Let $\A\in \CPEA_{\sf T}$ and $X\subseteq \A$. Then $X$ is said to be {\it a non--principal type} if $\prod X=0$.

\item Let $\kappa$ be a cardinal $\leq  2^{\omega}$. We say that a countable algebra $\A\in \CPEA_{\sf T}$  {\it admits a $\kappa$ omitting types theorem} if
whenever $\lambda <\kappa$ and  $(X_i: i<\lambda)$ a family of non--principal types, then there are a countable
algebra $\B\in \Gp_{\sf T}$ and an isomorphism $f:\A\to \B$, such that
$\bigcap_{x\in X_i}f(x)=\emptyset$ for each $i\in \lambda$. 

\item An algebra $\A$ generated by $\beta$ has the {\it interpolation property with respect to $\beta$}, or simply
the interpolation property,
if for all non--empty sets $X_1, X_1\subseteq \beta$, for all $a,b\in A$, whenever $a\in \Sg^{\A}X_1$ and $b\in \Sg^{\A}X_2$
are such that $a\leq b$, then there exists $c\in \Sg^{\A}(X_1\cap X_2)$ satisfying $a\leq c\leq b$. 
\end{enumarab}
\end{definition}
\begin{theorem}\label{interpolation}  Let $\sf T$ be a countable rich subsemigroup of $({}^{\omega}\omega, \circ)$. 
In items (2)-(6) 
we assume that $\T$ is strongly rich.
\begin{enumarab}
\item Every algebra in $\sf CPEA_{\T}$ is representable,

\item Every countable atomic algebra in $\sf CPEA_{\T}$ is completely representable. Furthermore, the condition of countability cannot be omitted,

\item Every countable algebra in $\sf CPEA_{\T}$ admits a $\mathfrak{p}$ omitting types theorem,

\item  Every countable simple algebra (has no proper ideals) in $\sf CPEA_{\T}$ admits a $\sf cov K$
omitting types theorem, 

\item The statement of `omitting $< {}2^{\omega}$ many types' is independent from the axioms of set theory.
Assuming Martin's axiom, then for any cardinal $\lambda<2^{\omega}$
every countable algebra in $\sf CPEA_{\T}$
admits a $\lambda$  omitting types theorem,

\item The free algebras having $\leq \omega$ free generators have the interpolation property.
\end{enumarab}

In particular, the corresponding logic
enjoys the omitting types theorem, the Craig interpolation theorem, the Beth definability property
and a Vaught's theorem, namely, countable atomic theories have atomic models.
\end{theorem}
\begin{proof}

(1)  The proof that every {\it countable} algebra $\A\in \sf CPEA_{\T}$ 
is representable can be easily discerned below 
the surface of the proofs of item (3) and the last item proving interpolation. 
Here strong richness is not needed because one takes the non--principal type $X=\{0\}$ which is plainly preserved in any 
$\omega$--dilation $\B$ of $\A$ in the sense that $\prod^{\B}X=0$. We do not need that $\A=\Nr_{\omega}\B$, $\A\subseteq \Nr_{\omega}\B$ is enough.
Proving representability of countable algebras suffices, because $\CPEA_{\T}$ is a variety. 
So if $\A\in \sf CPEA_{\T}$, then by the downward \ls--Tarski theorem, 
$\A$ has an elementary countable subalgebra  which is representable by the above, 
so $\A$ is representable, too, since representability is preserved under elementary equivalence. 

(2) The second part is like the proof of theorem \ref{main}, using the second part in lemma \ref{dl} by
undergoing the obvious modifications, namely, restricting
everything to to be countable, the given algebra and the  new countable signature. Here strong richness is needed, so that the sum of co--atoms in the algebra $\A$ 
is the same as its sum in the $\omega$--dilation $\B$. Both sums are the (common) top element, in symbols, 
$\sum^{\A}\At\A=\sum^{\B}\At\A=1$, when (by strong richness) $\A=\Nr_{\omega}\B$.
By additivity of (admissable) substitutions, we will have $\sum \s_{\tau}\At\A=1$ 
for any such (admissable) $\tau$. The rest of the proof is identical to the proof of theorem \ref{main}. The proof 
of this item is also a special case of the proof of the next one when we consider the one non--principal type consisting of 
co--atoms.

To show that the countability condition cannot be dispensed with, we show that there are atomic (uncountable) algebras in $\CPEA_\T$ that are not completely representable.
It clearly suffices to show that the class of completely representable $\sf CPEA_{\T}$s, $\bold K$ for short, 
is not elementary, because atomicity is a first order property. We do this using a cardinality argument. In fact, what we show 
is more than needed.   Using exactly the same argument in \cite{HH}, 
one first shows that if $\C\in \K$ and 
$\C\models {\sf d}_{01}<1$, then $\At\C=2^{\omega}$. The argument is as follows: 
Suppose that $\C\models {\sf d}_{01}<1$. Then there is $s\in h(-{\sf d}_{01})$ so that if $x=s_0$ and $y=s_1$, we have
$x\neq y$. Fix such $x$ and $y$. For any $J\subseteq \omega$ such that $0\in J$, set $a_J$ to be the sequence with
$i$th co-ordinate is $x$ if $i\in J$, and is $y$ if $i\in \omega\sim J$.
By complete representability every $a_J$ is in $h(1^{\C})$ and so it is in
$h(x)$ for some unique atom $x$, since the representation is an atomic one.
Let $J, J'\subseteq \omega$ be distinct sets containing $0$.   Then there exists
$i<\omega$ such that $i\in J$ and $i\notin J'$. So $a_J\in h({\sf d}_{0i})$ and
$a_J'\in h (-{\sf d}_{0i})$, hence atoms corresponding to different $a_J$'s with $0\in J$ are distinct.
It now follows that $|\At\C|=|\{J\subseteq \omega: 0\in J\}|=2^{\omega}$.

Take $\D\in {\sf CPEA}_{\T}$ with universe $\wp({}^{\omega}2)$ and with operations defined the usual way (as in set algebras). 
Then $\D\models {\sf d}_{01}<1$ and plainly $\D\in \K$. 
Using the downward \ls--Tarski theorem, take a countable  elementary subalgebra $\B$ of $\D$.
This is possible because the signature of $\sf CPEA_{\T}$ is countable.
Then in $\B$ we have $\B\models {\sf d}_{01}<1$ because $\B\equiv \C$. But $\B$ is not completely 
representable, because if it were then by the above argument, we get that $|\At\B|=2^{\omega}$,
which is impossible because $\B$ is countable. We have $\D\in \bold K$, $\B\notin \bold K$ and 
$\D\equiv \B$, thus $\bold K$ is not elementary.

(3)  For the third item, we can assume by the second part of lemma \ref{dl}, 
using  strong richness of $\T$,  that $\A=\Nr_{\omega}\B$ where  $\B$ is an $\omega+\omega$--dimensional
dilation provided by the lemma. We can further assume that $A$ generates $\B$, and so $B$ is countable because both $A$ and the signature
are countable.
Fix a cardinal $\lambda< \mathfrak{p}$. We are given a family
of non--principal types $(X_i: i<\lambda)$ that we want to omit.
Now we work like in theorem \ref{main},  but instead of dealing with only one non--principal type,
namely, the type consisting of co--atoms, we now have $\lambda$--many types to omit, and $\lambda$ can well be uncountable, because it is consistent
that $\omega_1<\mathfrak{p}=2^{\omega}$ \cite{Fre}.
So here we can (and will) appeal to the Baire category lied to a reduction of such types
to countably many using the properties of $\mathfrak{p}$.

We show that for any non--zero $a\in \A$, there exist a countable $\D_a\in {\sf Gp}_{\T}$ having a countable base, and a
homomorphism (that is not necessarily injective)
$f_a:\A\to \D_a$,  such that $f_a(a)\neq 0$ and for all $i<\lambda$, $\bigcap_{x\in X_i} f_a(x)=\emptyset$.
From these $f_a$'s ($a\in A)$ we obtain the
required isomorphism that omits the given family of non--principal types by taking the
product $\D={\bf P}_{a\in A}\D_a$ which is countable since the index set $A$ and each $\D_a$ are  countable.
One then defines
$f:\A\to \D$ by $f(a)=(f_a(a): a\in \A)$.
Then $f$ is clearly injective because if $a\in \A$, then $f_a(a)\neq 0,$ hence $f(a)\neq 0$. Furthermore, it is of course onto $f(\A)\subseteq \D$, and it 
omits the given family of $\lambda$ non--principal types.
Because a subalgebra of a set algebra is a set algebra, then $f(\A)$ will give the required representation omitting the given set of non--principal
types, via $f$.

For the sake of brevity, let $\alpha=\omega+\omega$. Let $\sf adm$ be  the set of admissible substitutions in $\T$. In the present context 
$\tau\in \sf adm$ if $\dom\tau\subseteq \omega$ and $\rng\tau\cap \omega=\emptyset$.  Since $\T$ is countable, we have
$|\sf adm|\leq\omega$. In fact, we have $|\sf adm|=\omega$ because  $\sf adm$ contains any finite function $f:\alpha\to \alpha$
such that $\dom f\subseteq \omega$ and $\rng f\cap \omega=\emptyset$.  Indeed, for any $n\in \omega$ define the function $f_n:\alpha\to \alpha$  
by $f_n(n)=\omega$ 
and $f_n(i)=i$ otherwise. Then $\dom f_n=\{n\}$, $\rng f_n=\{\omega\}$ and clearly for $n\neq m$, $f_n\neq f_m,$ because they have different 
domains.
Then we have, as in theorem \ref{main}, for all $i< \alpha$, $p\in \B$ and $\sigma\in \sf adm$,
\begin{equation}\label{tarek1}
\begin{split}
\s_{\sigma}{\sf c}_{i}p=\sum_{j\in \alpha} \s_{\sigma}{\sf s}_j^ip.
\end{split}
\end{equation}
By $\A=\Nr_{\omega}\B$ we also have, for each $i<\lambda$, $\prod^{\B}X_i=0$, since $\A$ is a complete subalgebra of $\B$ as proved in theorem \ref{main}.
Because substitutions are completely additive, we have
for all $\tau\in \sf adm$ and all $i<\lambda$,
\begin{equation}\label{t1}
\begin{split}
\prod {\sf s}_{\tau}^{\B}X_i=0.
\end{split}
\end{equation}
For better readability, for each $\tau\in \sf adm$, for each $i\in \lambda$, let
$$X_{i,\tau}=\{{\sf s}_{\tau}x: x\in X_i\}.$$
Then by complete additivity, we have:
\begin{equation}\label{t2}\begin{split}
(\forall\tau\in {\sf adm})(\forall  i\in \lambda)\prod {}^{\B}X_{i,\tau}=0.
\end{split}
\end{equation}
Let $S$ be the Stone space of the Boolean part of $\B$, and for $x\in \B$, let $N_x$
denote the clopen set consisting of all
Boolean ultrafilters that contain $x$.
Then from (\ref{t1}) and (\ref{t2}), it follows that for $x\in \B,$ $j<\alpha,$ $i<\lambda$ and
$\tau\in \sf adm$, the sets
$$\bold G_{\tau,j,x}=N_{{\sf s}_{\tau}{\sf c}_jx}\setminus \bigcup_{i} N_{{\sf s}_{\tau}{\sf s}_i^jx}
\text { and } \bold H_{i,\tau}=\bigcap_{x\in X_i} N_{{\sf s}_{\tau}x}$$
are closed nowhere dense sets in $S$.
Also each $\bold H_{i,\tau}$ is closed and nowhere
dense.
Let $$\bold G=\bigcup_{\tau\in \sf adm}\bigcup_{i\in \alpha}\bigcup_{x\in B}\bold G_{\tau, i,x}
\text { and }\bold H=\bigcup_{i\in \lambda}\bigcup_{\tau\in  \sf adm}\bold H_{i,\tau}.$$
By the definition of $\mathfrak{p}$, $\bold H$ is meager, since it is a $\lambda$--union of nowhere dense sets and a 
nowhere dense set is obviously meager. 
By the definition of $\mathfrak{p}$, $\bold H$ is a countable collection of nowhere dense sets.
By the Baire Category theorem  for compact Hausdorff spaces,
we get that $X=S\smallsetminus \bold H\cup \bold G$ is dense in $S$,
since $\bold H\cup \bold G$ is meager, because $\bold G$ is meager, too, since
$\sf adm$, $\alpha$ and $\B$ are all countable.
Accordingly, let $F$ be an ultrafilter in $N_a\cap X$, then by (5), and 
definition,  $F$ is perfect.

Factor out the set $\Gamma=\{i\in \alpha:\exists j\in \omega: {\sf c}_i{\sf d}_{ij}\in F\}$ by the congruence
relation $k\sim l$ iff ${\sf d}_{kl}\in F.$ Then like in the proof of theorem \ref{main} 
$\omega\subsetneq \Gamma$ because
$\c_i{\sf d}_{ij}=1$ for all $i<j\in \alpha$, and 
$(\forall i\in \omega)(\exists j\not\in \alpha\sim \omega)({\sf d}_{ij}\in F$).   Let $M=\Gamma/\sim$. 
We define the representation function $f_a$ as in the proof of theorem \ref{main} using the thereby
obtained perfect ultrafilter
$F$ which contains $a$ and is outside $\bold H\cup \bold G$ as follows.
For $\tau: \omega\to \Gamma$, let  $\bar{\tau}:\omega\to M$ be defined via: 
$i\mapsto \tau(i)/\sim$ ($i\in \omega$). 
Let $V=\{\bar{\tau}\in {}^{\omega}M: \tau:\omega\to \Gamma: \omega\cap \rng(\tau)=\emptyset\}$.
Then $V\neq \emptyset$ because $\omega\subsetneq \Gamma$.
For $x\in A$, let $f_a(x)=\{\bar{\tau}\in V: \s_{\tau^+}^{\B}x\in F\}$ where $\tau^+=\tau\cup Id_{\alpha\sim \omega}\in {\sf adm}$.
This map, as before, is well defined and
the target (representing) set algebra  has countable base  $M$. Also $Id\in f_a(a)$ since $a\in F$,  
hence $f_a(a)\neq 0$. 

Observe that $Id\in \sf adm$ because $\dom(Id)=\rng(Id)=\emptyset$.
Showing that $f_a$ preserves cylindrifiers is exactly 
like in the proof of theorem \ref{main} by using that $F\notin \bold G$. The 
preservation of the other operations is straightforward to check using that the substitution operations are Boolean endomorphisms,
and that for $\tau, \sigma \in \T$, $\s_{\tau}\circ \s_{\sigma}=\s_{\tau\circ \sigma}$, 
so $f_a$ is a homomorphism. For omitting the given family 
of non--principal types, we use that 
$F$ is outside $\bold H$, too.  This means (by definition) that for each $i<\lambda$ and each  $\tau\in  \sf adm$
there exists $x\in X_i$, such that
${\sf s}_{\tau}^{\B}x\notin F$. Let $i<\lambda$. If $\bar{\tau}\in V\cap \bigcap _{x\in X_i}f_a(x)$, 
then $\s_{\tau^+}^{\B}x\in F$ which  is impossible because $\tau^+\in \sf adm$. 
We have shown that for each $i<\lambda$, $\bigcap_{x\in X_i}f_a(x)=\emptyset.$
Thus $f_a$ omits  the given family of non--principal types, and we are done. By the special case 
proved in previous item here countability is essential as well.

(4)  Now in case $\A$ is simple, we can prove a stronger result. Simplicity of $\A$ means that the corresponding theory is complete.
Assume that $\A$ is simple, and let $(X_i: i<\lambda)$ be the given family of non--principal types with $\lambda<\sf covK$.
By lemma \ref{dl}, let $\B$ be an $\omega$--dilation of $\A$ such that $\A=\Nr_{\omega}\B$.
Define $\bold H$ and $\bold G$ like in the previous item. By the properties of $\sf covK$, the union $\bold H\cup \bold G\subseteq S$ (where $S$ is the Stone space of $\B$)
does not cover $S$, because $S$ is a Polish space.
Accordingly, we know that there is a Boolean perfect ultrafilter  $F$ of $\B$ in $S\sim \bold H\cup \bold G$.
Define $f$ as in the previous item via the perfect ultrafilter $F.$

Let $a\in A$ be non--zero. Then  $a\in F$ or $-a\in F$ hence $Id\in f(a)$ or $Id\in f(-a)$. It follows that $f\neq 0$
and by simplicity it is injective. The preservation
of the required meets and joins
is proved exactly like above.  Here we do not guarantee that
$S\sim \bold H\cup \bold G$ is dense which was the case in the previous item. So if $\A$ were not simple and $a\neq 0$, then we might not find a perfect ultrafilter containing
$a$ as in the last item, for $N_a\cap (S\sim \bold H\cup \bold G)$ could  well be empty. 
Recall that it is consistent that $\sf covK<\mathfrak{p}$, so that 
this result is stronger than that proved in the previous item 
when restricted to simple algebras.

(5) The independence is proved similarly to \cite[Theorem 3.2.8]{Sayed} together with the fact that
both cardinals $\sf covK$ and $\mathfrak{p}$ can be forced (using iterated forcing)
to be any uncountable regular cardinal between $\omega$ and  $2^{\omega}$.
Martin's axiom ($\sf MA$) forces that $\mathfrak{p}=\sf cov K=2^{\omega}$, so using  item (3) above, we get the required.

One can give a more direct proof that does not depend on forcing. 
It is known \cite[Theorem 1, pp. 492]{ma}, that $\sf MA$ implies that if $\lambda<2^{\omega}$ and
$X$ is a compact and Hausdorff space satisfying the countable chain condition $(\sf ccc)$, then the $\lambda$ union of nowhere dense sets in $X$
is a countable  union. Any second countable topological space satisfies the $\sf ccc$. In particular, the Stone space of the given countable algebra 
satisfies the $\sf ccc$. Thus, assuming $\sf MA$, $\bold H\cup \bold G$
as defined above, is a countable union.
An application of the
Baire Category theorem
finishes the proof.

(6)    Let $\A$ be the free $\CPEA_{\T}$ on $\omega$ generators.  We show that $\A$ has the interpolation property.
Let $\B\in \CPEA_{\omega+\omega}$, such that $\A=\Nr_{\omega}\B$
and $A$ generates $\B$. Note that both $\A$ and $\B$ are countable. Such
dilations exist like before by lemma \ref{dl} since $\T$ is strongly rich.
For the sake of brevity, again let $\alpha$ denote $\omega+\omega$.
This time we will use two perfect ultrafilters
to build two representations giving the required result. We proceed contrapositively.

Assume that $X_1, X_2\subseteq \A$, $a\in \Sg^{\A}X_1$ and $c\in \Sg^{\A}X_1$ such that $a\leq c$.
We want to find an interpolant, that is, we want to find $b\in \Sg^{\A}(X_1\cap X_2)$, such that $a\leq b\leq c$.  Assume for contradiction that there is no interpolant in $\A$.
Then we claim that there is no such interpolant in $\B$. Here we use $\A=\Nr_{\omega}\B$ (so strong richness of $\T$ is essential). 
Indeed, if there is an interpolant of $a$ and $c$ in $\B$, $b$ say,
then by the last part of lemma \ref{dl} the set $J=\Delta c\cap (\alpha\sim \omega)$ will be finite, hence one can cylindrify the indices
in  $J$ in any fixed order
obtaining the  interpolant $\c_{(J)}b\in \Nr_{\omega}\B=\A$. 
To see why $\c_{(J)}b$ is an interpolant of $a\leq c$ in $\A$, observe that cylindrification 
on $J$ does not alter $a$ nor $c$, because they are in the neat
reduct $\A$,  so that  we have $a=\c_{(J)}a\leq \c_{(J)}b\leq \c_{(J)}c=c$.

Now the non--existence of an interpolant of $a$ and $c$ even in the (bigger) $\omega$--dilation $\B$,
will eventually lead to a contradiction, as we proceed to show.
Arrange $\adm\times\alpha \times \Sg^{\B}X_1$
and $\adm\times\alpha\times \Sg^{\B}X_2$
into $\alpha$-termed sequences:
$\langle (\tau_i,k_i,x_i): i\in \alpha\rangle\text {  and  }\langle (\sigma_i,l_i,y_i):i\in \alpha\rangle,
\text { respectively.}$
Thus we can define by recursion (or step-by-step)
$\alpha$-termed sequences of witnesses:
$\langle u_i:i\in \alpha\rangle \text { and }\langle v_i:i\in \alpha\rangle$
such that for all $i\in \alpha$ we have:
$$u_i\in \alpha\smallsetminus
(\Delta a\cup \Delta c)\cup \cup_{j\leq i}(\Delta x_j\cup \Delta y_j\cup \dom\tau_j\cup \rng\tau_j\cup \dom\sigma_j\cup \rng\sigma_j)$$
$$\cup \{u_j:j<i\}\cup \{v_j:j<i\}$$
and
$$v_i\in \alpha\smallsetminus(\Delta a\cup \Delta c)\cup
\cup_{j\leq i}(\Delta x_j\cup \Delta y_j\cup \dom\tau_j\cup \rng\tau_j\cup \dom\sigma_j\cup \rng\sigma_j)$$
$$\cup \{u_j:j\leq i\}\cup \{v_j:j<i\}.$$
For an  algebra $\D$, we write ${\mathfrak{Bl}}\D$ to denote its Boolean reduct.
For a Boolean algebra $\C$  and $Y\subseteq \C$, we write
${\sf fl}^{\C}Y$ to denote the Boolean filter generated by $Y$ in $\C.$  Now let
\begin{align*}
Y_1&= \{a\}\cup \{-{\sf s}_{\tau_i}{\sf  c}_{k_i}x_i+{\sf s}_{\tau_i}{\sf s}_{u_i}^{k_i}x_i: i\in \alpha\},\\
Y_2&=\{-c\}\cup \{-{\sf s}_{\sigma_i}{\sf  c}_{l_i}y_i+{\sf s}_{\sigma_i}{\sf s}_{v_i}^{l_i}y_i:i\in \alpha\},\\
H_1&= {\sf fl}^{\mathfrak{Bl}\Sg^{\B}(X_1)}Y_1,\  H_2={\sf fl}^{\mathfrak{Bl}\Sg^{\B}(X_2)}Y_2,\\ 
\end{align*}
and
$$H={\sf fl}^{\mathfrak{Bl}\Sg^{\B}(X_1\cap X_2)}[(H_1\cap \Sg^{\B}(X_1\cap X_2)
\cup (H_2\cap \Sg^{\B}(X_1\cap X_2)].$$
Then we claim that $H$ is a proper filter of $\Sg^{\B}(X_1\cap X_2).$
To prove this claim it is sufficient to consider any pair of finite, strictly
increasing sequences of natural numbers
$$\eta(0)<\eta(1)\cdots <\eta(n-1)<\alpha\text { and } \xi(0)<\xi(1)<\cdots
<\xi(m-1)<\alpha,$$
and to prove that the following condition holds:

(+) For any $b_0$, $b_1\in \Sg^{\B}(X_1\cap X_2)$ such that
$$a\cdot \prod_{i<n}(-{\sf s}_{\tau_{\eta(i)}}{\sf  c}_{k_{\eta(i)}}x_{\eta(i)}+{\sf s}_{\tau_{\eta(i)}}{\sf s}_{u_{\eta(i)}}^{k_{\eta(i)}}x_{\eta(i)})\leq b_0$$
and
$$(-c)\cdot \prod_{i<m}
(-{\sf s}_{\sigma_{\xi(i)}}{\sf  c}_{l_{\xi(i)}}y_{\xi(i)}+{\sf s}_{\sigma_{\xi(i)}}{\sf s}_{v_{\xi(i)}}^{l_{\xi(i)}}y_{\xi(i)})\leq b_1$$
we have
$$b_0\cdot b_1\neq 0.$$
This can be proved by a tedious induction on $n+m$. We only give the base of the induction.
If $n+m=0$, then (+) simply
expresses the fact that no interpolant of $a$ and $c$ exists in
$\Sg^{\B}(X_1\cap X_2).$
In more detail: if $n+m=0$, then $a_0\leq b_0$
and $-c\leq b_1$. So if $b_0\cdot b_1=0$, we get $a\leq b_0\leq -b_1\leq c.$

Proving that $H$ is a proper filter of $\Sg^{\B}(X_1\cap X_2)$,
let $H^*$ be a (proper Boolean) ultrafilter of $\Sg^{\B}(X_1\cap X_2)$
containing $H.$
We thereby obtain  ultrafilters $F_1$ and $F_2$ of $\Sg^{\B}X_1$ and $\Sg^{\B}X_2$,
respectively, such that
$$H^*\subseteq F_1,\ \  H^*\subseteq F_2$$
and (*)
$$F_1\cap \Sg^{\B}(X_1\cap X_2)= H^*= F_2\cap \Sg^{\B}(X_1\cap X_2).$$
Now for all $x\in \Sg^{\B}(X_1\cap X_2)$ we have
$$x\in F_1\iff x\in F_2.$$
Also  $F_i$ for $i\in \{1,2\}$ satisfy the following
condition:

(**) For all $k<\alpha$, for all $\tau\in {\sf adm}$ for all $x\in \Sg^{\B}X_i$
if ${\sf s}_{\tau}{\sf  c}_kx\in F_i$ then ${\sf s}_{\tau}{\sf s}_l^kx$ is in $F_i$ for some $l\notin \Delta x.$

So by definition  for each $i\in \{1, 2\}$, $F_i$ is perfect.
For $i\in \{1, 2\}$, one defines a homomorphism $f_i$ 
on the subalgebra 
$\Sg^{\A}X_i$ of $\A$ using the hitherto constructed perfect ultrafilter $F_i$ exactly like the proof of item (3) above.
Like before, using (**) the map $f_i$ is a well--defined non--zero homomorphism. Indeed, $f_1(a)\neq 0$ and $f_2(-c)\neq 0$. 
By (*) we get that $f_1$ and $f_2$ agree on the common part $\Sg^{\A}(X_1\cap X_2)$.

Without loss of any generality, 
we can assume that $X_1\cup X_2$ (freely) generates $\A$. Freeness of $\A$  enables us to paste these homomorphisms, to a single one $h$ say, having 
domain $\Sg^{\A}{(X_1\cup X_2)}=\A$ and satisfying that $h\upharpoonright \Sg^{\A}X_1=f_1$ and $h\upharpoonright \Sg^{\A}X_2=f_2$. Since $h(a)=f_1(a)$ and $h(-c)=f_2(-c)$, then we get,
by $Id\in f_1(a)$ (since $a\in F_1)$ and $Id\in f_2(-c)$ (since $-c\in F_2$), that $Id\in f_1(a)\cap f_2(-c)=h(a)\cap h(-c)=h(a\cdot -c)$, since $h$ is homomorphism.
Thus $h(a\cdot -c)\neq 0$ which contradicts that $a\leq c$, and we are done.
The metalogical consequences follow 
by applying standard `bridge theorems'
in abstract algebraic logic passing from the algebra side to the logic side
\end{proof}
The next theorem follows by crossing the bridge from the other side.

\begin{corollary} \label{supap} For any strongly rich semigroup $\T$ on $\omega$,
$\CPEA_{\T}$ has the super amalgamation property.
\end{corollary}
\section{Solution to the Finitizability problem for quantifier logics with equality}

We are tempted to
say that the non--commutative fragment of Keisler's logic with equality that we defined (algebraically) in the last section  is a reasonable solution to the 
{\it finitizability problem
in algebraic logic} for first order logic {\it with equality}.  
A satisfactory solution for first order logic {\it without equality} 
was provided by Sain \cite{Sain} with respect to Tarskian semantics. 

\subsection{Old solution with respect to Tarskian semantics}

For first order logic {\it with equality}
the following result, to the best of our knowledge,  is the best obtained so far. To formulate the result we need to recall some notation.
$\sf PEA_{\omega}$ stands for the class of polyadic equality algebras, $\sf QEA_{\omega}$ stands for the class of quasi--polyadic equality algebras where we have only finite substitutions,
$\sf RQEA_{\omega}$ stands for the class of representable $\sf QEA_{\omega}$s, all of dimension $\omega$, 
and finally $\Rd_{qea}$ denotes `quasi--polyadic equality reduct.'
\begin{theorem}\cite{SG, Sain} \label{Sain} There is a class $\sf K$ of $\omega$--dimensional set algebras with 
a finitely axiomatizable equational theory satisfying (1)--(3): 
\begin{enumarab}
\item  ${\sf K}=\bold S \Rd \sf PEA_{\omega},$ that is algebras in $\sf K$ are reducts of $\sf PEA_{\omega}$s,
\item ${\bf S}\Rd_{qea}{\sf K}=\sf RQEA_{\omega},$
\item $\bold H\sf K$ is a finitely based variety.
Here $\bold H$ is the operation of forming homomorphic
images.
\end{enumarab}
\end{theorem}
\begin{demo}{Sketch} We exhibit such a ${\sf K}$ but prove  (1) and (2) only, and we show `axiomatizability by a finite schemata' instead of finite axiomatizability.
If the semigroup $\T$ we work with is finitely presented, then the finite 
schemata of equations can be translated to an equivalent finite set of equations in a finite signature 
as explained in the first item of the next theorem. 

Let $\T$ be any one of the two strongly rich semigroups in example \ref{rich}; for definiteness let it be the first.
Let $\sf Set$ be the class of set algebras of the form 
$(\mathfrak{B}({}^{\omega}U), \c_i, {\sf d}_{ij}, \s_{\tau})_{i,j\in \omega, \tau\in \T}$, $U$ a non--empty set, 
and let ${\sf K}={\bf SP}\sf Set$. Then $\A\in \sf K$ $\iff$ it has top element a {\it generalized} cartesian 
space, which is a {\it disjoint union} of cartesian spaces. It is proved in \cite{Sain} that $\sf K$ is not a variety; $\sf K$ is not closed under $\bold H$ nor ${\bf Up}$ (ultraproducts),
so $\sf K$ is not even a quasi--variety. 

Let $\Sigma$ be the finite schemata of equations given in \cite{SG} which 
is taken in the same signature of $\sf CPEA_T$. Then  it is easy to see that $\sf K\models \Sigma$; this is a soundness theorem.
The converse 
$\bold H{\sf K}={\sf Mod}\Sigma$, a {\it weak} completeness theorem, is harder to prove. We omit the proof referring the reader to \cite{SG}.
The intrusion of $\bold H$ here means that the set of axioms in $\Sigma$ stipulated in the expanded 
signature of $\sf CPEA_T$,  enforce that the (old) quasi--polyadic equality operations are representable, 
but the axioms are not strong enough to enforce representability of the newly added substitution operations 
$\s_{\sf suc}$ and $\s_{\sf pred}$. More precisely, if $\A\models \Sigma$, then $\A\in \bold H\sf K$, 
so there exist a set algebra $\B$ (whose top element is a generalized cartesian space) in $\sf K$ 
and a surjective homomorphism $f:\B\to \A$.  But  $f$ 
{\it might not be injective}.  In other words, these  substitution operations, though represented faithfully in $\B$, 
may not stay representable in the quotient algebra 
$\B/{\sf ker f}$.  On the other hand,  the remaining ${\sf QEA}_{\omega}$ operations  are represented faithfully 
(meet  as 
intersection and cylindrifiers as projections, \ldots etc) in both $\B$ and $\B/{\sf ker f}$ because ${\sf RQEA}_{\omega}$ 
is a variety (so it is closed under $\bold H$ which is not the case with $\sf K$). 
So what we can (and will) show is that if $\A\models \Sigma$, then 
$\Rd_{qea}\A\in \sf RQEA_{\omega}$. 
If in addition $\A$ is countable,
we show that $\Rd_{qea}\A$ has a $\mathfrak{p}$ omitting types theorem, and if $\Rd_{qea}\A$ 
is simple then, it has a $\sf covK$ omitting types theorem.

For the first part, one uses the neat embedding argument in lemma \ref{dl} by iterating
the unary operation ${\sf s}_{\sf succ}$ but assuming commutativity of cylindrifiers \cite{Sain, AU}.
Then the quasi-polyadic equality reduct would be representable by
Henkin's neat embedding theorem for $\sf QEA$s.
For the second part on omitting types, one uses the same argument in the proof of the third and fourth items in
theorem \ref{interpolation} (since $\T$ is strongly rich). But here the proof is simpler, the meager set formed in the Stone topology corresponding to 
omitting the given family of non--principal types, is the double join $\bigcup_{i\in \lambda}\bigcup_{\tau\in \T} \bold H_{i,\tau},$
where $\lambda$ is the number of non--principal types to be omitted.  Here the second union is taken on the whole of the 
countable semigroup $\T$ not restricted only to $\sf adm$. The unit hitherto obtained is a set of the form $\bigcup_{i\in I}{} ^{\alpha}U^{(p)}$
where  $U_i=U$, $U$ a countable set,
for every $i\in I$ and $p\in
{}^{\alpha}{U}$ \cite{AU}. Such a space is called a {\it compressed space}.
\end{demo}
The metalogical interpretation of the third condition $\sf Mod(\Sigma)={\bold H}\sf K$, 
means that the algebraizable logic corresponding to $\sf K$ is complete with respect to validities  but {\it is not compact.}
It is the case that  $\models \phi\iff \vdash \phi$ for any formula $\phi$, but it can happen that 
there exists a set of formulas $\Gamma\cup \{\phi\}$, 
such that $\Gamma\models \phi$, but $\Gamma\nvdash \phi$.

\subsection{New solution with respect to relativized semantics}

Throughout this subsection $\T$ will denote a countable 
rich subsemigroup of $({}^\omega\omega, \circ)$. 
For first order logic with equality, the solution we propose of course depends on the choice of the semigroup
$\T$. In our solution we have $\sf K=\sf Gp_{\T}$, so we do not need $\bold H$ as formulated in the third item of the theorem \ref{Sain}. 
Here $\sf K$ {\it itself is a variety}. This gives that the corresponding logic is both complete and 
compact. Furthermore, $\sf K$ will be {\it finitely axiomatizable}.
The price we pay for such substantial improvements is that we relativize semantics.
We require that the top elements of representable algebras are arbitrary unions 
of certain spaces rather than {\it disjoint} unions of such spaces which was the case in the last theorem due to Sain.

We need some definitions before we formulate 
a series of properties of $\Gp_\T$ (the class of set algebras) 
some of which are new and some already proved.
\begin{enumarab}

\item $\sf V$ is {\it atom--canonical}, if whenever $\A\in \sf V$ and $\A$ is atomic,
then the complex algebra of the atom structure of $\A$, in symbols $\Cm\At\A$, is in $\sf V$.

\item $\At(\sf V)$ denotes the class of atom structures of atomic algebras in $\sf V$ and
${\sf Str}({\sf V})=\{\alpha\in \At({\sf V}):  \Cm(\alpha)\in \sf V\}$.

\item The {\it  \d\ completion} of a Boolean algebra with operators $\A$ is the unique (up to isomorphisms that fix $\A$ pointwise)  complete algebra
$\B$ such that $\A\subseteq \B$, and $\A$ is {\it dense} in $\B$, meaning that for all non-zero $b\in B$,  there  exists non-zero
$a\in A$  such that $a\leq b$. If $\A$ is atomic and completely additive, then its \d\ completion is $\Cm\At\A$.
\end{enumarab}
Having the needed definitions at hand, in the next theorem, we collect some of our previously proved statements,
together with some more addressing the new notions introduced above.
Proofs will be provided for the newly added statements. 

\begin{theorem}\label{fp}  Assume that $\T$ is a finitely presented 
rich subsemigroup of $({}^{\omega}\omega, \circ)$ with the finite set $\sf S$ presenting $\sf T$. Assume
that $\T$ has distinguished elements $\pi$ and $\sigma$ such that $\sigma\in \sf S$. Such a semigroup exists \cite{Sain}.
Then there is a recursive finite set of equations $\Sigma$, such that ${\sf Mod}(\Sigma)=\Gp_{\T}$.
If $T$ is a strongly rich semigroup (not necessarily finitely presented), then the 
properties in items (2)--(6) in theorem \ref{interpolation} hold for
the variety $\Gp_{\T}$, and some more. In more detail:
\begin{enumarab}
\item Every countable atomic $\Gp_{\T}$ is completely representable. 
For the corresponding logic $\L_{\T}$,
every countable atomic theory has an atomic model. Furthermore, this atomic model omits any family of
non--principal types,
\item $ {\sf Gp_{\T}}$ is canonical and atom--canonical,
\item $\At(\Gp_{\T})={\sf Str}(\Gp_{\T})$, each is first order definable,
and they generate $\Gp_{\T}$ in the strong sense,  that is, $\Gp_{\T}=\bold S\Cm{\sf Str}(\Gp_{\T})=\bold S\Cm\At(\Gp_{\T})$,
\item $\Gp_{\T}$ is closed under \d\ completions,
\item $\L_{\T}$ has an omitting types theorem,
\item $\L_{\T}$ has the Craig interpolation property,
the Beth definability property, and enjoys a Robinson joint consistency theorem. Consequently, ${\sf Gp}_{\T}$
has the super amalgamation property.
\end{enumarab}
\end{theorem}
\begin{proof}
We know from the first item of theorem \ref{interpolation} 
that if $\T$ is rich, then $\CPEA_\T=\bold I {\sf Gp}_\T$. If $\T$ is {\it rich and finitely presented}, then 
using exactly the techniques in \cite{Sain}
one can truncate the axiomatization of $\sf CPEA_\T$ given in
definition \ref{capa}, restricted to a
rich finitely presented semigroups, to be strictly finite and recursive.  This entails that the signature is also finite.
It turns out that  ${\sf Gp}_{\T}$ is term--definitionally equivalent to a variety
in a finite signature, namely, the Boolean operations together with 
$\{\c_0, {\sf d}_{01}, \s_{\tau}: \tau\in {\sf S}\}$, where ${\sf S}$ is a finite set presenting $\T.$ 
The idea here is that  the successor--like transformation 
$\sigma$  which, by hypothesis, is simultaneously one of the distinguished elements of $\T$ and is in the set $\sf S$,
generates the rest of the operations.

(1) follows by noting that atomic representations are complete ones.
The items that remain to be proved are
items  (2), (3), and (4).
Canonicity follows from the fact that equations axiomatizing
$\Gp_{\T}$ are  Sahlqvist \cite[Theorem 2.95]{book}.

First order axiomatizability of $\At(\Gp_{\T})$
follows from \cite[Theorem 2.84]{book}, and closure under \d\ completions follows from
\cite[Theorem 2.96]{book}, by noting that $\sf CPEA_{\T}$ is a conjugated variety.
By complete additivity, we have $\Cm\At\A$ is the \d\ completion of an atomic $\A$,
so we get that $\sf V=\Gp_{\T}$ is also
atom--canonical and closed under \d\ completions. This proves (4).
Hence by definition $\At(\sf V)={\sf Sr}(\sf V)$. By canonicity and atom--canonicity, conjuncted with \cite[Theorem 2.88]{book},
we get the last part in (3). Items (5) and (6) are proved in theorem \ref{interpolation} and we are done. 
\end{proof}

\section{Finite dimensional case}

There is a finite dimensional version of the finitizability problem in algebraic logic as well \cite{v, Nemeti, Fer, f, M, Simon2, Sain, Bulletin} which we
discuss in some depth, culminating in formulating the exact finite version of our main 
finitizability in theorem \ref{m} below.

\subsection{Local guarding and clique guarded semantics}

Here we study finite variable fragments of first order logic with different semantics,
which we call {\it local guarding}, allowing cylindrifiers to commute but {\it only locally}.  These semantics were studied by Hirsch and Hodkinson for relation algebras
\cite[Chapter 13]{book}. 
We start with defining certain semantical notions.
\begin{definition}
Assume that $1<m<n<\omega$. Let $M$ be a {\it relativized representation} of $\A\in \CA_m$, that is, there exists an injective
homomorphism $f:\A\to \wp(V)$, where $V\subseteq {}^mM$ and $\bigcup_{s\in V} \rng(s)=M$. Here we identify the set algebra with universe $\wp(V)$ with
its universe $\wp(V)$, since the concrete operations, like Boolean intersection or projections (cylindrifiers)
uniquely depend on the top element $V$. For $s\in V$ and $a\in \A$,
we write $M\models a(s)$ for $s\in f(a)$. Let  $\L(\A)^n$ be the first order signature using $n$ variables
and one $m$--ary relation symbol for each element in $A$. Then {\it an $m$--clique} is a set $C\subseteq M$ such
$(a_1,\ldots, a_{m-1})\in V=1^M$
for distinct $a_1, \ldots, a_{m}\in C.$
Let
${\sf C}^n(M)=\{s\in {}^nM :\rng(s) \text { is an $m$--clique}\}.$
Then ${\sf C}^n(M)$ is called the {\it $m$--Gaifman hypergraph of $M$}, with the $m$--hyperedge relation $1^M$.
\begin{enumarab}
\item The {\it clique guarded semantics $\models_c$} are defined inductively. For atomic formulas and Boolean connectives they are defined
like the classical case and for existential quantifiers
(cylindrifiers) they are defined as follows:
For $i<n$ and $\bar{s}\in {}^nM$, $M, \bar{s}\models_c \exists x_i\phi$ $\iff$ there is a $\bar{t}\in {\sf C}^n(M)$, $\bar{t}\equiv_i \bar{s}$ such that
$M, \bar{t}\models \phi$.
\item  We say that $M$  is  {\it $n$--square},
if witnesses for cylindrifiers can be found on $m$ cliques. More precisely,
whenever $\bar{s}\in {\sf C}^n(M), a\in \A$, $i<m$,
and   $l:m\to n$ is an injective map, if $M\models {\sf c}_ia(s_{l(0)},\ldots, s_{l(m-1)})$,
then there is a $\bar{t}\in {\sf C}^n(M)$ with $\bar{t}\equiv _i \bar{s}$,
and $M\models a(t_{l(0)}, \ldots, t_{l(m-1)})$.

\item  $M$ is said to be {\it $n$--flat} if  it is $n$--square and
for all $\phi\in \L(\A)^n$, for all $\bar{s}\in {\sf C}^n(M)$, for all distinct $i,j<n$,
$M\models_c [\exists x_i\exists x_j\phi\longleftrightarrow \exists x_j\exists x_i\phi](\bar{s}).$
\end{enumarab}
\end{definition}
This semantics is also a {\it relativization} to ${\sf C}^n(M)$, it is a {\it local relativization}.
By convention by an $\omega$--flat or $\omega$--square representation of a $\CA_n$ having countably many atoms, 
we mean an ordinary representation.
For terminolgy on neat reducts for $\CA$s, we follow \cite{HMT1}. Fix ordinals $m<n$. If $\B\in \CA_n$, then $\Nr_m\B(\in \CA_m)$ 
is the neat {\it $m$--reduct of $\B$.} 
If $\A\in \CA_m$ and $\A\subseteq \Nr_m\B$, with $\B\in \CA_n$, we say that $\B$ is an {\it $n$--dilation of $\A$}, or simply a {\it dilation} of $\A$ if $n$ is clear 
from context.  For $\sf K\subseteq {\sf CA}_n$, ${\sf Nr}_m{\sf K}=\{\Nr_m\B: \B\in {\sf K}\}\subseteq \CA_m$.

The semantical notion of having an $n$--flat representation is equivalent to the syntactical one of having an $n$--dilation, as expressed in 
the next (completeness) theorem
with respect to clique guarded semantics:
\begin{lemma}\label{neatsq}
Let $2<m<n<\omega$. Then an algebra $\A\in \CA_m$ has an $n$--flat representation $\iff \A\in \bold S{\sf Nr}_m\CA_n.$
\end{lemma} 
\begin{proof}
\cite[Theorem 13.46]{book}. Let $M$ be an $n$--flat representation of $\A$. We show that $\A\subseteq \Nr_m\D$, for some $\D\in \CA_n$.
For $\phi\in \L(\A)^n$,
let $\phi^{M}=\{\bar{s}\in {\sf C}^n(M):M,\bar{s}\models_c \phi\}$, where ${\sf C}^n(M)$ is the $m$--Gaifman hypergraph.
Let $\D$ be the algebra with universe $\{\phi^{M}: \phi\in \L(A)^n\}$ and with  cylindric
operations induced by the $m$-clique--guarded (flat) semantics read off the connectives. Then in $\D$, by $n$--flatness, 
cylindrifiers commute so $\D\in \CA_n$.  We identify $a\in \A$ with the $m$--ary relational formula $a(\bar{x})$ it defines in $\L(\A)^n$.
Define $\theta:\A\to \D$, via $a\mapsto a(\bar{x})^{M}$. Then exactly like in the proof of \cite[Theorem 13.20]{book},
$\theta$ is a neat embedding, that is, $\theta(\A)\subseteq \Nr_m\D$. 

The other direction is harder. We give an outline.  From an $n$--dilation $\D$ of the {\it canonical extension}
of $\A\in \CA_m$, one constructs an $n$--dimensional 
hyperbasis \cite[Definition 12.11]{book} modified to the $\CA$ case.
This $n$--dimensional hyperbasis can be viewed as a saturated set of $n$--dimensional hypernetworks (mosaics) 
that can be glued together in a step--by--step
manner to build the required representation of $\A$. For the relation algebra case witness
\cite[Lemmata 13.33-34-35, Proposition 36]{book}. 
\end{proof}

\begin{theorem}\label{2.12} For any $2<m<n<\omega$, the variety of algebras having $n+1$--flat representations
is not finitely axiomatizable over the variety of algebras having $n$--flat representations, and
$\sf RCA_m$ is not finitely axiomatizable over the variety of algebras in $\CA_m$ having
$n$--flat representations.
\end{theorem}
\begin{proof}
Assume that $3\leq m\leq n<\omega$. Using the notation in \cite{book}, let
$\mathfrak{C}(m,n,r)=\Ca(H_m^{n+1}(\A(n,r),  \omega))$
be as defined in \cite[Definition 15.4]{book}.
It can be proved that for any finite $k\geq 1$, $\C(m, m+k, r)\in {\sf Nr}_m\CA_{m+k}\sim \bold S{\sf Nr}_m\CA_{m+k+1}$
and $\Pi_{r/U}\C(m, m+k, r)\in \sf RCA_m$ where $U$ is any non--principal ultrafilter on $\omega$,
witness \cite[Corollary 15.10, Exercise 2, pp. 484]{book}. Using Lo\'s theorem \cite[proof of Theorem 15.1(4)]{book} and the fact that the
variety of algebras having $m<n\leq \omega$ flat representations coincides with the variety $\bold S{\sf Nr}_m\CA_n$ as proved in the previous lemma, 
we get the required result.
\end{proof}
The result that for $2<n<\omega$ and positive $k$, the variety $\bold S{\sf Nr}_n\CA_{n+k+1}$ is not finitely axiomatizable over the variety $\bold S{\sf Nr}_n\CA_{n+k}$ 
is lifted to the transfinite replacing `non--finite axiomatizability' by `not axiomatizable by a finite schemata' 
in \cite{t}. For more negative results on decidability and finite axiomatizability, we have: 
\begin{theorem}\label{decidable}
Let $n\geq 5$. Then it is undecidable to tell
whether a finite algebra in  $\CA_3$ has an $n$--flat representation, and the variety
$\bold S{\sf Nr}_3\CA_n$ cannot be finitely axiomatizable
in $k$th order logic for any positive $k$.
\end{theorem}
\begin{proof} This can be proved by lifting the analogous results for relation algebras
\cite[Theorem 18.13, Corollaries 18.14, 18.15, 18.16]{book}. One uses the construction of Hodkinson in \cite{Hodkinson}
which associates recursively to every atomic relation algebra $\sf R$, an atomic  $\A\in \CA_3$
such that ${\sf R}\subseteq \sf Ra\A$, 
the latter is the relation algebra reduct of $\A$, cf. \cite[Definition 5.3.7, Theorem 5.3.8]{HMT2}.
The idea for the second part is that the existence of any such finite axiomatization in 
$k$th order logic for any positive $k$,  
gives a decision procedure for telling whether a finite algebra is 
in $\bold S{\sf Nr}_3\CA_n$
or not \cite{book}, which is impossible by the first part.
\end{proof}
\subsection{Omitting types in clique guarded semantics}

We will show that the omitting types theorem fails for  locally guarded fragments, in the sense stated in our next result, theorem \ref{OTT}.
A different proof is given in theorem \ref{OTT2}. The proofs will be used below to generalize classical results proved by Hirsch and Hodkinson \cite{Hodkinson, HH}.
We need some preparation. Throughout this subsection, unless otherwise indicated, $m$ is a finite ordinal $>2$.

\begin{definition}
Let $T$ be a first order theory in a signature using $m<\omega$ many variables and $Fm$ be the set of formulas in this signature.
Assume that $2<m<n\leq \omega$. The non--empty set $M$ is {\it an $n$--flat model of $T$} if $M$ is an
$n$--flat representation of $\Fm_T$, where the last is the {\it Tarski--Lindenbaum} $\CA_m$
of formulas of  dimension $m$ corresponding
to $T$ formed the usual way. Let $\Gamma\subseteq Fm$ be a set of formulas, referred to as {\it a type}.
Then  $\Gamma$ is {\it omitted in $M$}, if there exists an
isomorphism $f:\Fm_T\to \wp(V)$ where $\bigcup_{s\in V}\rng(s)=M$ such that
$\bigcap_{\phi\in \Gamma}f(\phi_T)=\emptyset$,
otherwise $\Gamma$ is {\it realized in $M$}.
\end{definition}
We need the notions of atomic networks and atomic games \cite{book, HHbook2}:

\begin{definition}\label{games}  
\begin{enumarab}
\item An {\it atomic network} on an atomic algebra $\A\in \CA_m$  is a map $N: {}^m\Delta\to  \At\A$, where
$\Delta$ is a non--empty set of {\it nodes}, denoted by $\nodes(N)$, satisfying the following consistency conditions: 
\begin{itemize}
\item If $\bar{x}\in {}^m\nodes(N)$, and $i<j<m$, then $N(\bar{x})\leq {\sf d}_{ij}\iff x_i=x_j$.
\item If $\bar{x}, \bar{y}\in {}^m\nodes(N)$, $i<m$ and $\bar{x}\equiv_i \bar{y}$, then  $N(\bar{x})\leq {\sf c}_iN(\bar{y})$.
\end{itemize}
\item   Assume that $\A\in \CA_m$ is  atomic and that $n, k\leq \omega$. 
The {\it atomic game $G^n_k(\At\A)$, or simply $G^n_k$}, is the game played on atomic networks
of $\A$
using $n$ nodes and having $k$ rounds \cite[Definition 3.3.2]{HHbook2}, where
\pa\ is offered only one move, namely, {\it a cylindrifier move}: 
\begin{itemize}
\item Suppose that we are at round $t>0$. Then \pa\ picks a previously played network $N_t$ $(\nodes(N_t)\subseteq n$), 
$i<m,$ $a\in \At\A$, $\bar{x}\in {}^m\nodes(N_t)$, such that $N_t(\bar{x})\leq {\sf c}_ia$. For her response, \pe\ has to deliver a network $M$
such that $\nodes(M)\subseteq n$,  $M\equiv _i N$, and there is $\bar{y}\in {}^m\nodes(M)$
that satisfies $\bar{y}\equiv _i \bar{x}$, and $M(\bar{y})=a$. 
\end{itemize}
\item  We write $G_k(\At\A)$, or simply $G_k$, if $n\geq \omega$.
The {\it atomic game $F^n(\At\A)$, or simply $F^n$}, is like $G^n_{\omega}(\At\A)$ except that
\pa\ has the option to re--use the available
$n$ nodes during the play.
\end{enumarab}
\end{definition}

We let  $\bold S_c$ denote the operation of forming {\it complete} subalgebras.

\begin{lemma}\label{flat} Let $2<m<n$. If $\A\in \bold S_c{\sf Nr}_m\CA_{n}$ is atomic,  then  \pe\ has a \ws\ in $F^{n}$. 
In particular, if $\A$ is finite and \pa\ has a \ws\ in $F^{n}$, then $\A\notin \bold S{\sf Nr}_m\CA_{n}$.
\end{lemma}
\begin{proof} \cite[Theorem 33]{r}.
\end{proof}

For $2<m<n<\omega$, $n$--square {\it complete} representations are defined exactly
like the classical case. In particular, any such representation is atomic. 

 \begin{lemma}\label{square} Let $2<m<n<\omega$ and $\A\in \CA_m$.
Then $\A$ has a complete  $n$--square representation $\iff$  
\pe\ has a \ws\ in $G_{\omega}^n(\At\A).$ 
\end{lemma}
\begin{proof}
We prove $\implies$ which is all what we need. Let $M$ be a complete  $n$--square representation of $\A$.
One proceeds like in the proof of lemma \ref{neatsq}, but using $L_{\infty,\omega}$ formulas in the signature $\L(\A)^n$ to build the required 
dilation.
Construct an $n$--dimensional dilation $\D$ with universe 
${\sf C}^{n}(M)$ and operations induced by clique guarded semantics by defining for $s\in {\sf C}^n(M)$, 
and $\phi_i\in \L(\A)^n$ ($i\in I\neq \emptyset$), $M, s\models _c \bigwedge_{i\in I} \phi_i\iff M,s\models_c \phi_i$ for all $i\in I$.
Then $\D$ will be an atomic $\CA_n$ \cite[Item (3), Theorem 13.20]{book}; here infinite conjunctions are used. 

For each $\bar{a}\in 1^{\D},$ define \cite[Definition 13.22] {book} a labelled
hypergraph $N_{\bar{a}}$ with nodes $n$, and
$N_{\bar{a}}(\bar{x})$ when $|\bar{x}|=m$, is the unique atom of $\A$
containing the tuple
$(a_{x_0},\ldots, a_{x_{1}},\ldots, a_{x_{m-1}}, a_{x_0}\ldots,\ldots a_{x_0})$ of length $n>m.$
If $s\in 1^{\D}$ and $i, j<n$,
then $s\circ [i|j]\in 1^{\D}$. By \cite[Lemma 13.24]{book} $N_{\bar{a}}$ is a
network with $\nodes(N)=n$.
Let $H$ be the {\it symmetric closure}
of $\{N_{\bar{a}}: \bar{a}\in 1^{\D}\}$, that is $\{N\theta: \theta: m\to m, N\in H\}$. (Here $N\theta$ is defined by 
$N\theta(x_0,\ldots, x_{m-1})=N(\theta(x_m),\ldots, \theta(x_{m-1})\}$).
Then $H$ is an $n$--dimensional basis generalized to $\CA$s the obvious way \cite[Lemma 13.26]{book}. 

Now \pe\ can win $G_{\omega}^n$ by always
playing a subnetwork of a network in the constructed basis $H$.
In round $0$, when \pa\ plays
the atom $a\in \A$, \pe\ chooses $N\in H$ with $N(0,1,\ldots, m-1)=a$ and plays $N\upharpoonright m$.
In round $t>0$, inductively if the current network is $N_{t-1}\subseteq M\in H$, then no matter how \pa\ defines $N$, we have
$N\subseteq M$ and $|N|<n$, so there is $z<n$, with $z\notin \nodes(N)$.
Assume that  \pa\ picks $x_0,\ldots, x_{m-1}\in \nodes(N)$, $a\in \At\A$ and $i<m$, such that
$N(x_0,\ldots, x_{m-1})\leq {\sf c}_ia$, so $M(x_0, \ldots,  x_{m-1})\leq {\sf c}_ia$.
But  $H$ is an $n$--dimensional basis, so there is $M'\in H$ with $\nodes(M')\subseteq n$, such that
$M'\equiv _i M$ and $M'(x_0, \ldots, z, \ldots,  x_{m-1})=a$, with $z$ in the $i$th place.
Now \pe\ responds with $N_t=M'\upharpoonright \nodes(N)\cup \{z\}$.
\end{proof}

{\bf Rainbow constructions:} In our next theorem we use a rainbow construction so we need to review some notions and terminology. 
Let $2<m<\omega$. The most general exposition of $\CA_m$ rainbow constructions is given
in \cite[Section 6.2,  Definition 3.6.9]{HHbook2} in the context of constructing atom structures from classes of models.

Our models are just {\it coloured graphs} \cite{HH} which are complete graphs whose edges are labelled by the rainbow colours, $\g$ (greens), $\r$ (reds), and 
$\w$ (whites) satisfying certain consistency conditions. The greens are 
$\{\g_i: 1\leq i< m-1\}\cup \{\g_0^i: i\in \sf G\}$  and the reds are $\{\r_{ij}: i,j \in \sf R\}$ where
$\sf G$ and $\sf R$ are two relational structures. The whites are $\w_i: i\leq m-2$.
In coloured graphs the following triangles are forbidden:
\begin{eqnarray*}
&&\nonumber\\
(\g, \g^{'}, \g^{*}), (\g_i, \g_{i}, \w_i)
&&\mbox{any }1\leq i\leq  m-2,  \\
(\g^j_0, \g^k_0, \w_0)&&\mbox{ any } j, k\in \sf G,\\
\label{forb:match}(\r_{ij}, \r_{j'k'}, \r_{i^*k^*})&&\mbox{unless }i=i^*,\; j=j'\mbox{ and }k'=k^*.
\end{eqnarray*}
Also, in coloured graphs some $m-1$ tuples (hyperedges) are also labelled by shades of yellow $\y_S$ (some $S\subseteq \sf G$) \cite[4.3.3]{HH}. 
We follow verbatim \cite{HH} for rainbow constructions.  We recall the definition of cones which are special coloured graphs.
\begin{definition}\label{cone}
Let $i\in {\sf G}$, and let $M$ be a coloured graph consisting of $m$ nodes
$x_0,\ldots,  x_{m-2}, z$. We call $M$ {\it an $i$ - cone} if $M(x_0, z)=\g_0^i$
and for every $1\leq j\leq m-2$, $M(x_j, z)=\g_j$,
and no other edge of $M$
is coloured green.
$(x_0,\ldots, x_{m-2})$
is called  the {\it base of the cone}, $z$ the {\it apex of the cone}
and $i$ the {\it tint of the cone}.
\end{definition}
The atoms of a rainbow atom structure of dimension $m$ are equivalence classes of surjective maps $a:m\to \Delta$, where $\Delta$ is a coloured graph
in the rainbow signature, and the equivalence relation relates two such maps $\iff$  they essentially define the same graph \cite[4.3.4]{HH};
the nodes are possibly different but the graph structure is the same. We let $[a]$ denote the equivalence class containing $a$.
The accessibility binary relation corresponding 
to the $i$th  cylindrifier $(i<m)$ is defined by:  $[a] T_i [b]\iff a\upharpoonright m\sim \{i\}=b\upharpoonright m\sim \{i\},$ 
and the accessibility unary relation corresponding to the $ij$th diagonal element ($i<j<m$) is defined by: $[a]\in D_{ij}\iff a(i)=a(j)$.

For rainbow atom structures, there is a one to one correspondence between atomic networks and coloured graphs \cite[Lemma 30]{HH}, 
so for $2<m<n\leq \omega$, we use the graph versions of the games $G^n_k$, $k\leq \omega$, and  $F^n$ played on rainbow atom 
structures of dimension $m$ \cite[pp.841--842]{HH}. Recall that $F^n$ is like $G_{\omega}^n$ \cite[pp.841]{HH}, 
but now played on coloured graphs where \pa\ has the bonus to re-use the $n$ nodes in play.

A \ws\  for either player in the graph version of $G^n_k$
is dictated by  a \ws\ for the same player in a simple private 
\ef\ forth game having $n'\leq n$ pairs of pebbles and $k'\leq k$ rounds.  
This game, denoted below by ${\sf EF}_{k'}^{n'}(\sf G, R)$, is played on the 
two relational structures 
$\sf G$ and $\sf R$ \cite[Definition 16.2, Theorem  16.5]{book}.
In sharp (and interesting) contrast to the omitting 
types theorem proved in theorem \ref{interpolation}, we have:

\begin{theorem}\label{OTT}
Let $2<m<\omega$ and $n\in \omega$ such that $n\geq m+3$. Then there exists a countable, atomic  and 
complete theory $T$ using $m$ variables, with consequence relation defined semantically,
that is $\Fm_T$ is an atomic (countable) $\sf RCA_m$, such that if $\Gamma$ is the non--principal type consisting of co-atoms; $\Gamma=\{\neg \phi: \phi_T\in \At\Fm_T\}$,
then  $\Gamma$ is not omitted in any $n$--square model, {\it a fortiori}, in any $n$--flat one.
\end{theorem}
\begin{demo}{A fairly complete sketch}  Throughout the proof $m$ is fixed to be a finite ordinal $>2$.
Let $\A$ be an atomic, countable and simple ${\sf RCA}_m$, such that its \d\ completion $\Cm\At\A$ is not in $\bold S{\sf Nr}_m\CA_{m+3}$.
Such an algebra exists as we proceed to show.

{\bf Idea:} The argument used  is a combination of the rainbow construction in \cite{Hodkinson} which is implemented model--theoretically, 
together with the blow up and blur construction used in \cite{ANS}. The idea  is  to embed a finite (rainbow) algebra $\D\notin \bold S{\sf Nr}_m\CA_{m+3}$ 
in the \d\ completion of an atomic (infinite) algebra $\A\in {\sf RCA}_m$, where $\A$ 
is obtained by blowing up and blurring $\D$. The `blowing up' is done by splitting some of the atoms of $\D$ each into
infinitely many atoms (of  $\A$). The term `blur' refers to the fact that the algebraic structure of $\D$ is {\it blurred} at the level of $\A$, $\D$ {\it does not embed} into $\A$.
However, the algebraic structure of $\D$ is {\it not blurred} at the `global level of $\Cm\At\A$', 
because $\D$ embeds into $\Cm\At\A$. 

The proof  of the existence of $\A$ as alleged is divided into three parts. In the first part we blow up and blur a finite rainbow algebra $\D$, 
denoted below by $\CA_{m+1, m}$, 
by splitting some of the atoms (the red ones), 
each into infinitely many, getting a weakly representable atom structure $\At$.
This means  that the  term algebra on $\At$, which is the subalgebra of the complex algebra $\Cm\At$ generated by the atoms, in symbols $\Tm\At$,
is representable.
In the second part, we embed $\D$ into  $\Cm\At$, which is the \d\ completion of $\Tm\At$.
In the third part, we show that  \pa\ has a \ws\ in  $G^{m+3}(\At\D)$, hence {\it a fortiori} in the game 
$F^{m+3}(\At\D)$ (where he is allowed to re-use the $m+3$ nodes in play). This,  together with lemmata \ref{flat}
and \ref{square}, imply that $\D\notin
\bold S{\sf Nr}_m\CA_{m+3}$ and that $\D$ has no $m+3$--square representation. Since $\D$ embeds into $\Cm\At$, 
we conclude that  $\Cm\At$ is outside $\bold S{\sf Nr}_m\CA_{m+3}$ and it  has no $m+3$--square representation, as well. 
In particular, $\Cm\At$ is not representable, hence $\At$ is {\it not strongly representable}
obtaining the result in 
\cite{Hodkinson} as a special case. Now we give the details:

{\bf (1) Blowing up and blurring  $\CA_{m+1, m}$ forming a weakly representable atom structure $\At$:}
Take the finite rainbow cylindric algebra $R(\Gamma)$
as defined in \cite[Definition 3.6.9]{HHbook2},
where $\Gamma$ (the reds) is taken to be the complete irreflexive graph $m$, and the greens
are  $\{\g_i:1\leq i<m-1\}
\cup \{\g_0^{i}: 1\leq i\leq m+1\}$ so tht $\sf G$ is the complete irreflexive graph $m+1$.
Call this finite rainbow $m$--dimensional cylindric algebra, based on ${\sf G}=m+1$ and ${\sf R}=m$, $\CA_{m+1, m}$ and denote its atom structure by $\At_f$.
One  then replaces each  red colour
used in constructing  $\CA_{m+1, m}$  by infinitely many with superscripts from $\omega$, 
getting a weakly representable atom structure $\At$, that is,
the term algebra $\Tm\At$ is representable.
The resulting atom structure (with $\omega$--many reds),  call it $\At$, 
is the rainbow atom structure that is like the atom structure in \cite{Hodkinson} except that we have $m+1$ greens
and not infinitely many as is the case in \cite{Hodkinson}.

Everything else is the same. In particular, the rainbow signature \cite[Definition 3.6.9]{HHbook2} now consists of $\g_i: 1\leq i<m-1$, $\g_0^i: 1\leq i\leq m+1$,
$\w_i: i<m-1$,  $\r_{kl}^t: k<l< m$, $t\in \omega$,
binary relations, and $m-1$ ary relations $\y_S$, $S\subseteq m+1$.
There is a shade of red $\rho$; the latter is a binary relation that is {\it outside the rainbow signature}.
But $\rho$ is used as a  label for  coloured graphs built during a `rainbow game', and in fact, \pe\ can win the rainbow $\omega$--rounded game
and she builds an $m$--homogeneous (coloured graph) model $M$ by using $\rho$ when
she is forced a red \cite[Proposition 2.6, Lemma 2.7]{Hodkinson}.
Then $\Tm\At$ is representable as a set algebra with unit $^mM$; this can be proved exactly as in \cite{Hodkinson}.
 By $m$--homogeneity, is meant that every coloured graph of size $\leq m$ embeds into $M$, and that 
such coloured graphs are uniquely determined by
their isomorphism types, regardless of their
location in $M$.

Having $M$ at hand, one constructs  an atomic $m$--dimensional set algebras based on $M$.
In more detail, let
$W = \{ \bar{a} \in {}^m M : M \models ( \bigwedge_{i < j <m} \neg \rho(x_i, x_j))(\bar{a}) \},$
and for $\phi\in L_{\infty, \omega}^m$, let
$\phi^W=\{s\in W: M\models \phi[s]\}.$
Here $W$ is the set of all $m$-ary assignments in 
$^mM$, that have no edge labelled by $\rho$.
Let $\A$  be the relativized set algebra with domain
$\{\varphi^W : \varphi \,\ \textrm {a first-order} \;\ L_m-
\textrm{formula} \}$  and unit $W$, endowed with the algebraic
operations ${\sf d}_{ij}, {\sf c}_i, $ etc., $(i<m)$ in the standard way, and formulas are taken in the 
rainbow signature only (without $\rho$).

Classical semantics for $L_m$ rainbow formulas and their
semantics by relativizing to $W$ coincide.
That is if $\cal S$ is the  set algebra with domain  $\wp ({}^{m} M )$ and
unit $ {}^{m} M$, then the map $h : \A
\longrightarrow \cal S$ given by $h:\varphi ^W \longmapsto \varphi^{M}=\{ \bar{a}\in
{}^{m} M: M \models \varphi (\bar{a})\}$ is both 
well--defined and an injective homomorphism \cite[Proposition 3.13]{Hodkinson}.
This depends essentially on \cite[Lemma 3.10]{Hodkinson} which says that any permutation $\chi$ of $\omega\cup \{\rho\}$,
$\Theta^{\chi}$ as defined in  \cite[Definitions 3.9, 3.10]{Hodkinson} 
is an $m$--back 
and--forth system.    
The system $\Theta^{\chi}$ consists of $\chi$ isomorphisms between coloured graphs having the same size $\leq m$ in the following sense.
Let $\chi$ be a permutation of the set $\omega \cup \{ \rho\}$. Let
$ \Gamma, \triangle$ be coloured graphs that have the same size, and let $ \theta :
\Gamma \rightarrow \triangle$ be a bijection. We say that $\theta$
is a $\chi$-\textit{isomorphism} from $\Gamma$ to $\triangle$ if for
each distinct $ x, y \in \Gamma$,
if $\Gamma ( x, y) = \r_{jk}^i$,
then  $\triangle( \theta(x),\theta(y)) =\r_{jk}^{\chi(i)}$ if $\chi(i) \neq \rho$ 
and is equal to $\rho$  otherwise.

If $\Gamma ( x, y) = \rho$, then $\triangle( \theta(x),\theta(y))=\r_{jk}^{\chi(\rho)}$ if $\chi(\rho) \neq \rho$
and is equal to $\rho$ otherwise.
Finally,  $\Gamma(x,y)$ is not red then $\Delta(\theta(x), \theta(y))=\Gamma(x, y).$
One uses such $m$--back--and--forth systems mapping a 
tuple $\bar{b} \in {}^m M \backslash W$ to a tuple
$\bar{c} \in W$ preserving any formula containing the non-red symbols that are
`moved' by the system, so if $\bar{b}\in {}^mM$ refutes the $L_m$ rainbow formula  $\phi$, then there is a $\bar{c}$ in $W$ 
refuting $\phi$, so  the set algebra $\A$ is embeddable in $\cal S$.

Since $\A$ is in $\bold I{\sf Cs}_m$, then $\A$ is simple. But $\Tm\At\subseteq \A$ (they have the same atom structure),  
then $\Tm\At$ is simple and representable, too  (a subalgebra of a simple algebra is simple).
The algebras $\Tm\At\subseteq \A\subseteq \Cm\At$ share the same atom structure. Also,  $\Cm\At=\Cm\At\A$ is the \d\ completion of the other two (countable) algebras. 
The atoms of all three algebras consist of 
(equivalence classes) of surjections $a:m\to \Gamma$, $\Gamma$ a coloured graph, 
whose edges are not labelled by $\rho$; here only the rainbow colours corresponding to the above infinite rainbow signature are used. In the formula algebra
$\A$ such atoms are expressed semantically by 
so--called $\sf MCA$ formulas \cite[Definition 4.3]{Hodkinson}.

{\bf (2) Embedding $\CA_{m+1, m}$ into the \d\ completion of $\Tm\At$:} 
We embed $\CA_{m+1, m}$ into  the complex algebra $\Cm\At$, the \d\ completion of $\Tm\At$.
Let ${\sf CRG}_f$ denote  the class of coloured graphs on 
$\At_f$ and $\sf CRG$ be the class of coloured graph on $\At$. We 
can assume that  ${\sf CRG}_f\subseteq \sf CRG$.

Write $M_a$ for the atom that is the (equivalence class of the) surjection $a:m\to M$, $M\in \sf CGR$.
Here we identify $a$ with $[a]$; no harm will ensue.
We define the (equivalence) relation $\sim$ on $\At$ by
$M_b\sim N_a$, $(M, N\in {\sf CGR}):$
\begin{itemize}
\item $a(i)=a(j)\Longleftrightarrow b(i)=b(j),$

\item $M_a(a(i), a(j))=\r^l\iff N_b(b(i), b(j))=\r^k,  \text { for some $l,k$}\in \omega,$

\item $M_a(a(i), a(j))=N_b(b(i), b(j))$, if they are not red,

\item $M_a(a(k_0),\dots, a(k_{m-2}))=N_b(b(k_0),\ldots, b(k_{m-2}))$, whenever
defined.
\end{itemize}
We say that $M_a$ is a {\it copy of $N_b$} if $M_a\sim N_b$. 
We say that $M_a$ is a {\it red atom} if it has at least one edge labelled by a red rainbow colour $\r_{ij}^l$ for some $i<j<m$ and $l\in \omega$. 
Clearly every red atom $M_a$ has infinitely countable many red copies, which we denote by $\{M_a^{(j)}: j\in \omega\}$.
Now we define a map $\Theta: \CA_{m+1, m}=\Cm{\At_f}$ to $\Cm\At$,
by  specifing  first its values on ${\sf At}_f$,
via $M_a\mapsto \sum_jM_a^{(j)}$; each atom maps to the suprema of its 
copies.  If $M_a$ is not red,   then by $\sum_jM_a^{(j)}$,  we understand $M_a$.
This map is extended to $\CA_{m+1, m}$ the obvious way by $\Theta(x)=\bigcup\{ \Theta(y):y\in \At\CA_{m+1, m}, y\leq x\}$. The map
$\Theta$ is well--defined, because $\Cm\At$ is complete. 
It is not hard to show that the map $\Theta$ 
is an injective homomorphim. 
Injectivity follows from the fact that $M_a\leq f(M_a)$, hence $\Theta(x)\neq 0$ 
for every atom $x\in \At(\CA_{m+1, m})$.
Now we check the preservation of diagonal elements and cylindrifiers. 
\begin{itemize}
\item Diagonal elements: Let $i< j<m$ and $x:m\to \Gamma$, $\Gamma\in \sf CGR$. 
Then:
$$M_x\leq \Theta({\sf d}_{ij}^{\Cm\At_f})\iff\ 
M_x\leq \sum \bigcup_{a_i=a_j}M_a^{(j)}\iff 
M_x\leq \bigcup_{a_i=a_j}\sum_j M_a^{(j)}$$
$$\iff  M_x=M_a^{(j)}  \text { for some $a: m\to M$, $a(i)=a(j)$}
\iff M_x\in {\sf d}_{ij}^{\Cm\At}.$$

\item Cylindrifiers: Let $i<m$. By additivity of cylindrifiers, we restrict our attention to atoms 
$M_a\in \At_f$ with $a:m\to M$, and $M\in \sf CRG_f\subseteq \sf CRG$. Then: 
$$\Theta({\sf c}_i^{\Cm\At_f}a)=\Theta (\bigcup_{[c]\equiv_i[a]} M_c)
=\bigcup_{[c]\equiv_i [a]}\Theta(M_c)
=\bigcup_{[c]\equiv_i [a]}\sum_j M_c^{(j)}$$
$$=\sum_j \bigcup_{[c]\equiv_i [a]}M_c^{(j)}
=\sum _j{\sf c}_i^{\Cm\At}M_a^{(j)}
={\sf c}_i^{\Cm\At}(\sum_j M_a^{(j)})
={\sf c}_i^{\Cm\At}\Theta(a).$$
\end{itemize}
{\bf (3) A \ws\ for \pe\ in $G^{m+3}\At(\CA_{m+1, m})$:}
One first shows that \pa\ has a \ws\ in the \ef\ forth pebbled game 
 ${\sf EF}_r^p( m+1, m)$ \cite[Definition 16.2]{HHbook2}, 
played on the complete irreflexive graphs $m+1$ and $m$ since $m+1$ is `longer'. Here $r$ is the number of rounds and $p$ is the number of pairs of pebbles
on board. Using $p>m$ many pairs of pebbles \pa\ can win this game in $m+1$ many rounds.
In each round $0,1\ldots m$, \pe\ places a  new pebble  on  a new element of $m+1$.
The edge relation in $m$ is irreflexive so to avoid losing
\pe\ must respond by placing the other  pebble of the pair on an unused element of $m$.
After $m$ rounds there will be no such element, so she loses in the next round.
This game lifts to a graph game  \cite[pp.841]{HH} on $\At_f$ which  in this
case is equivalent to the graph version of $G^{m+3}$.

Now  \pa\ can lift his \ws\ in ${\sf EF}^{m+1}_{m+1}(m+1, m)$ to the graph game $G^{m+3}\At(\CA_{m+1, m})$ on $\At_f$ as follows: 
Like in \cite{HH}, using the notation in {\it op.cit}, \pa\ bombards \pe\ with cones have the same base and green tints,
forcing \pe\ to play an inconsistent triple of reds, that is a red triangle whose indices do not match. 
In his zeroth move, \pa\ plays a graph $\Gamma$ with
nodes $0, 1,\ldots, m-1$ and such that $\Gamma(i, j) = \w_0 (i < j <
m-1), \Gamma(i, m-1) = \g_i ( i = 1,\ldots, m-2), \Gamma(0, m-1) =
\g^0_0$, and $ \Gamma(0, 1,\ldots, m-2) = \y_{m+1}$. This is a $0$-cone
with base $\{0,\ldots, m-2\}$. In the following moves, \pa\
repeatedly chooses the face $(0, 1,\ldots, m-2)$ and demands a node
$\alpha$ with $\Phi(i,\alpha) = \g_i$, $(i=1,\ldots, m-2)$ and $\Phi(0, \alpha) = \g^\alpha_0$,
in the graph notation -- i.e., an $\alpha$-cone, $\alpha\leq m+2$,  on the same base.
\pe\ among other things, has to colour all the edges
connecting new nodes created by \pa\ as apexes of cones based on the face $(0,1,\ldots, m-2)$. By the rules of the game
the only permissible colours would be red. Using this, \pa\ can force a
win, using $m+3$ nodes.

Then by lemma \ref{flat} this implies that  $\CA_{m+1, m}\notin
\bold S{\sf Nr}_m\CA_{m+3}$. Since $\CA_{m+1,m}$ embeds into $\Cm\At$, 
hence $\Cm\At$  is outside 
$\bold S{\sf Nr}_m\CA_{m+3}$, too. Also by lemma \ref{square}, the finite algebra $\D$, hence $\Cm\At$, 
does not have an $m+3$--square representation, because
\pa\ has a \ws\ in $G^{m+3}(\At\D)$ and $\D$ embeds into $\Cm\At$.

Using the algebra $\A$ (or $\Tm\At\A)$, we are now ready to prove the failure 
of the omitting types theorem as stated in the next  
statement highlighted in bold, thereby proving the theorem.

{\bf The non--principal type of co--atoms of $\Tm\At$ cannot be omitted in an $m+3$--square model:}
First, we claim that  $\A=\Tm\At$ has no complete $m+3$--square representation. Assume for contradiction that $M$
is a complete $m+3$--square representation of $\A$. Hence there exists an injective homomorphism $g:\A\to \wp(V)$
where $V\subseteq {}^mM$ and $\bigcup_{s\in V}\rng(s)=M$ and since $g$ is  also an atomic $m+3$--square 
representation, then $\bigcup_{x\in \At\A}g(x)=V$.
Observe  that $\At\A=\At\Cm\At\A$. Accordingly, one can define $f:\Cm\At\A\to \wp(V)$
by $f(a)=\bigcup_{x\in \At\A, x\leq a} g(x)$ $(a\in \Cm\At\A)$.
Then $f$ induces an $m+3$--square  representation of $\Cm\At\A,$ so
$\CA_{m+1, m}$ has an $m+3$--square representation, too, since it embeds in $\Cm\At\A$.
But this is impossible by lemma \ref{square}, because as shown above \pa\ has a \ws\ in $G^{m+3}_{\omega}(\At\CA_{m+1, m})$
(in only finitely many rounds)
and an $m+3$--square representation of $\CA_{m+1, m}$ is plainly a complete 
one.

Now we prove 
the theorem. We can identify $\A$ with $\Fm_T$ for some countable, consistent and complete atomic
theory $T$ using $m$ variables.  The theory $T$ is consistent because $|A|>1$, $T$ is complete because $\A$ is simple, and 
$T$ is atomic because $\A$ is atomic. 
Let $\Gamma=\{\neg\phi: \phi_T\in \At\Fm_T\}$. Then $\Gamma$ is non--principal, because 
$\prod^{\A}\{-a: a\in \At\A\}=-(\sum^{\A} \At\A)=-1=0$. We claim that the non--principal type $\Gamma$ cannot be omitted
in any $m+3$--square model. Assume for contradiction that it can. Then there is a non--zero homomorphism 
$f:\Fm_T\to \wp(V)$ where $V=1^M$ and $M$ is an $m+3$--square
representation of $\Fm_T$,  such that $\bigcap_{\phi\in \Gamma}f(\phi_T)=\emptyset$. Since $\Fm_T$ is simple, then $f$ is an embedding.
Now $1^M= -\bigcap_{\phi\in \Gamma}f(\phi_T)=\bigcup_{\phi\in \Gamma}f(-\phi_T)=\bigcup _{x\in \At\A}f(x)$.
Thus $\bigcup_{x\in \At\A}f(x)=1^M=V$, so $M$ is an atomic, hence complete $m+3$--square representation of $\A$, 
which is impossible, and we are done.
\end{demo}

\begin{remark} Fix $2<m <\omega$. We proved that $\bold S{\sf Nr}_m\CA_{m+3}$ is not closed under \d\ completions, since $\Cm\At$ 
is the \d\ completion of the representable
algebra $\Tm\At$. 
Now the argument used above works {\it uniformly} for any ordinal $k\geq 3$  (possibly infinite), that is, 
for the variety $\bold S{\sf Nr}_m\CA_{m+k}$. For infinite $k$, by $m+k$ we mean ordinal addition, so 
that $m+k=k$.
The dimension $m+k$ is controlled by the number of greens $\sf num(g)$ that we start off with.  
One takes ${\sf num({\sf g})}=m+k-2$, so that $m+k=2+{\sf num(\sf g)}$. The number $2$ is the {\it increase in the number} 
from passing from {\it the number of `pairs of pebbles'} 
used in the private \ef\ forth game ${\sf EF}_{m+k-2}^{m+k-2}(m+k-2, m)$  to the {\it number of nodes} used in 
coloured graphs during the play lifted to the rainbow algebra $\D=\CA_{m+k-2, m}$. 
The last game is the graph version of $G_r^{m+k}(\At\D)$ (some $r\leq \omega)$.

In all cases \pa\  has a \ws\ in both games 
excluding an $m+k$--square representation of $\D$. If $k$ is finite, then $\D$ is finite and the number of rounds $r$ in both games is finite, that is, 
$m+k-2\leq r<\omega$. If $k$ is infinite, then $\sf num(g)=\omega$, $r=m+k-2=\omega$ and $\D=\CA_{\omega, m}$ is infinite. 
In both cases (finite and infinite), the rainbow algebra $\D$ embeds into the complex algebra of the atom structure obtained by {\it blowing up and bluring} $\At\D$, denoted above (when $k=3)$ 
by $\sf At$. The algebra $\D$ is {\it not blurred} in $\Cm\At$. It can be proved exactly like in \cite[Lemma 5.3]{Hodkinson} that $\Cm\At\cong \E$ via $X\mapsto \bigcup X$, 
where $\E$ is the {\it relativized 
non--representable set  algeba} with top element $W$ and universe $\{\phi^W: \phi\in L_{\infty,\omega}^n\}$ with $W$ as defned in the previous proof, 
$\phi$ is taken in the rainbow signature, and the operations defined the usual way like in cylindric set algebras relativized to $W$. 
The result in \cite{Hodkinson} is the special case when $k=\omega$. The embedding of $\D$ into $\Cm\At$ defined (using the notation in the above proof) 
via $M_a\mapsto \sum_j M_a^{(j)}$  does not work if the target algebra is $\Tm\At$, because 
$\Tm\At$ is not complete. Indeed, we do know that there can be {\it no embedding} from $\D$ into 
$\Tm\At$ because the latter is representable while the former is not; $\D$ {\it was blurred} in $\Tm\At$.  
\end{remark}

The following corollary follows immediately from the construction in theorem \ref{OTT}. It substantially strengthens the result in \cite{Hodkinson}.
\begin{corollary}\label{hod} Let $2<m<n\leq \omega$ and assume that $n\geq m+3$. Then the variety $\bold S{\sf Nr}_m\CA_n$, which is  the class 
of $\CA_m$s having $n$--flat representations, is not atom--canonical.
In particular, it is not closed under \d\ completions and, being conjugated, it is not Sahlqvist axiomatizable.
\end{corollary}
Next, we reprove theorem \ref{OTT} (for flatness) differently. We will use rainbows again, but we will be slightly more sketchy. 
Our construction here is inspired by the rainbow construction used 
for relation algebras in \cite{r} and the rainbow construction for cylindric algebras used in \cite{HH}.  But first a lemma.
\begin{lemma}\label{complete} Let $2<m<\omega$. 
For an atomic algebra 
$\A\in \CA_m$, $\A$ has a complete $n$--flat representation $\iff \A\in \bold S_c{\sf Nr}_m\CA_n.$
\end{lemma}
\begin{proof}  \cite[Theorem 13.45]{book}. The proof is similar to the proof of lemma \ref{neatsq} except that 
now we consider {\it complete} $n$--flat representations and the operation $\bold S_c$ (of forming {\it complete} subalgebras) 
instead of $\bold S$ (forming subalgebras) \cite[Proposition 36]{book}. 
\end{proof}
We give a different proof to theorem \ref{OTT} in the case of flatness:
\begin{theorem}\label{OTT2} Let $2<m<\omega$. Then there exists an atomic algebra $\A\in {\sf RCA}_m$ having countably many atoms
such that $\A\notin \bold S_c{\sf Nr}_m\CA_{m+3}$,  but such that $\A$ is elementary equivalent to a countable completely
representable $\CA_m$.  Furthermore, $\A$ can be used to violate the omitting types theorem  
as formulated in theorem  \ref{OTT} with respect to $m+3$--flat semantics.
\end{theorem}
\begin{proof} 
Fix finite $m>2$. Let $\A$ be the $m$--dimensional
rainbow cylindric algebra $R(\Gamma)$  \cite[Definition 3.6.9]{HHbook2}
where $\Gamma=\omega$, so that the reds ${\sf R}$ is the set $\{\r_{ij}: i<j<\omega\}$ and the greens
constitute the set ${\sf G}=\{\g_i:1\leq i <m-1\}\cup \{\g_0^i: i\in \Z\}$. In complete coloured graphs the forbidden triples are like 
in usual rainbow constructions \cite{HH} (as described above), 
but now we impose a new
forbidden triple in coloured graphs connecting two greens and one red. We stipulate that 
the triple  $(\g^i_0, \g^j_0, \r_{kl})$ is forbidden if $\{(i, k), (j, l)\}$ is not an order preserving partial function from
$\Z\to\N$. Here we identify $\omega$ with $\N$. 
The $m$--dimensional complex algebra of this atom structure, which we denote by $\A$ is based 
on the two ordered structure $\Z$ (greens) and $\N$ (reds).

In the present context the newly added forbidden triple 
makes it harder for \pe\ to win the  
game $G_k(\At\A)$ $(k\in \omega$) but not impossible.  
Indeed, it can  (and will) be shown that \pe\ has a \ws\ in $G_k(\At\A)$ for all $k\in \omega$.
Hence, using ultrapowers and an elementary chain argument  \cite[Theorem 3.3.5]{HHbook2}, one gets a countable algebra $\B$ 
such that $\B\equiv \A$, and  \pe\ has a \ws\ in $G_{\omega}(\At\B)$. Then $\B$, being countable, 
is completely representable by \cite[Theorem 3.3.3]{HHbook2}. 
On the other hand, we will show that \pa\ has a \ws\ in $F^{m+3}(\At\A)$, implying by lemma \ref{flat} 
that $\A\notin \bold S_c{\sf Nr}_m\CA_{m+3}$.

{\bf \pe's strategy in $G_k(\At\A)$ where $0<k<\omega$ is the number of rounds:}  
Let $0<k<\omega$. We proceed inductively. Let $M_0, M_1,\ldots, M_r$, $r<k$ be the coloured graphs at the start of a play of $G_k$ just before round $r+1$.
Assume inductively, that \pe\ computes a partial function $\rho_s:\Z\to \N$, for $s\leq r:$

\begin{enumroman}
\item $\rho_0\subseteq \ldots \rho_t\subseteq\ldots\subseteq\ldots  \rho_s$ is (strict) order preserving; if $i<j\in \dom\rho_s$ then $\rho_s(i)-\rho_s(j)\geq  3^{k-r}$, where $k-r$
is the number of rounds remaining in the game,
and 
$$\dom(\rho_s)=\{i\in \Z: \exists t\leq s, \text { $M_t$ contains an $i$--cone as a subgraph}\},$$

\item for $u,v,x_0\in \nodes(M_s)$, if $M_s(u,v)=\r_{\mu,k}$, $\mu, k\in \N$, $M_s(x_0,u)=\g_0^i$, $M_s(x_0,v)=\g_0^j$,
where $i,j\in \Z$ are tints of two cones, with base $F$ such that $x_0$ is the first element in $F$ under the induced linear order,
then $\rho_s(i)=\mu$ and $\rho_s(j)=k$.
\end{enumroman} 
For the base of the induction \pe\ takes $M_0=\rho_0=\emptyset.$ 
Assume that $M_r$, $r<k$  ($k$ the number of rounds) is the current coloured graph and that \pe\ has constructed $\rho_r:\Z\to \N$ to be a finite order preserving partial map
such conditions (i) and (ii) hold. We show that (i) and (ii) can be maintained in a 
further round.
We check the most difficult case. Assume that $\beta\in \nodes(M_r)$, $\delta\notin \nodes(M_r)$ is chosen by \pa\ in his cylindrifier move,
such that $\beta$ and $\delta$ are apprexes of two cones having
same base and green tints $p\neq  q\in \Z$. 
Now \pe\ adds $q$ to $\dom(\rho_r)$ forming $\rho_{r+1}$ by defining the value $\rho_{r+1}(p)\in \N$ 
in such a way to preserve the (natural) order on $\dom(\rho_r)\cup \{q\}$, that is maintaining property (i).
Inductively, $\rho_r$ is order preserving and `widely spaced' meaning that the gap between its elements is
at least $3^{k-r}$, so this can be maintained in a further round.

Now \pe\  has to define a (complete) coloured graph 
$M_{r+1}$ such that $\nodes(M_{r+1})=\nodes(M_r)\cup \{\delta\}.$ 
In particular, she has to find a suitable 
red label for the edge $(\beta, \delta).$
Having $\rho_{r+1}$ at hand she proceeds as follows. Now that $p, q\in \dom(\rho_{r+1})$, 
she lets $\mu=\rho_{r+1}(p)$, $b=\rho_{r+1}(q)$. The red label she chooses for the edge $(\beta, \delta)$ is: (*)\ \  $M_{r+1}(\beta, \delta)=\r_{\mu,b}$.
This way she maintains property (ii) for $\rho_{r+1}.$  Next we show that this is a \ws\ for \pe. 

{\bf Checking that \pe's strategy is a winning one:}  We check consistency of newly created triangles proving that $M_{r+1}$ is a coloured graph completing the induction. 
Since $\rho_{r+1}$ is chosen to preserve order, no new forbidden triple (involving two greens and one red) will be created.
Now we check red triangles only of the form $(\beta, y, \delta)$ in $M_{r+1}$ $(y\in \nodes(M_r)$). 
We can assume that  $y$ is the apex of a cone with base $F$ in $M_r$ and green tint $t$, say,
and that $\beta$ is the appex of the $p$--cone having the same base. 
Then inductively by condition (ii), taking $x_0$ to be the first element of $F$, and taking  
the nodes $\beta, y$, and the tints $p, t$, for $u, v, i, j$,  respectively, we have by observing that 
$\beta, y\in \nodes(M_r)$, $\beta, y\in \dom(\rho_r)$ and $\rho_r\subseteq \rho_{r+1}$, 
the following:  
$M_{r+1}(\beta,y)=M_{r}(\beta, y)=\r_{\rho_{r}(p), \rho_{r}(t)}=r_{\rho_{r+1}(p), \rho_{r+1}(t)}.$
By  her strategy, we have  $M_{r+1}(y,\delta)=\r_{\rho_{r+1}(t), \rho_{r+1}(q)}$ 
and we know by (*) that $M_{r+1}(\beta, \delta)=\r_{\rho_{r+1}(p), \rho_{r+1}(q)}$. 
The triple $(\r_{\rho_{r+1}(p), \rho_{r+1}(t)}, \r_{\rho_{r+1}(t), \rho_{r+1}(q)}, \r_{\rho_{r+1}(p), \rho_{r+1}(q)})$
of reds is consistent (witness forbidden triples of reds right before definition \ref{cone}) and we are done with this case. 
All other edge labelling and colouring $m-1$ tuples in $M_{r+1}$ 
by yellow shades are  exactly like in \cite{HH}.

{\bf  \pa\ can win the $\omega$-rounded game $F^{m+3}(\At\A)$:}  The idea here is that the newly added triple
 forces \pe\ to play reds $\r_{ij}$ with one of the indices forming a decreasing sequence in $\N$
in response to \pa\ playing cones having a common 
base and distinct green tints (demanding a red label for edges between 
appexes of two succesive cones.) Having the option to reuse 
the $m+3$ nodes is crucial for \pa\ to implement his \ws\ because he uses {\it finitely many} nodes to win an {\it infinite} $\omega$--rounded game.
The argument used is essentially the $\CA$ analogue of \cite[Theorem 33, Lemma 41]{r}.

In the initial round \pa\ plays a graph $M$ with nodes $0,1,\ldots, m-1$ such that $M(i,j)=\w_0$
for $i<j<m-1$
and $M(i, m-1)=\g_i$
$(i=1, \ldots, m-2)$, $M(0, m-1)=\g_0^0$ and $M(0,1,\ldots, m-2)=\y_{\Z}$. This is a $0$ cone.
In the following move \pa\ chooses the base  of the cone $(0,\ldots, m-2)$ and demands a node $m$
with $M_2(i,m)=\g_i$ $(i=1,\ldots, m-2)$, and $M_2(0,m)=\g_0^{-1}.$
\pe\ must choose a label for the edge $(m+1,m)$ of $M_2$. It must be a red atom $r_{nk}$, $n, k\in \N$. Since $-1<0$, then by the `order preserving' condition 
we have $n<k$.
In the next move \pa\ plays the face $(0, \ldots, m-2)$ and demands a node $m+1$, with $M_3(i,m)=\g_i$ $(i=1,\ldots, m-2)$,
such that  $M_3(0, m+2)=\g_0^{-2}$.
Then $M_3(m+1,m)$ and $M_3(m+1, m-1)$ both being red, the indices must match, so $M_3(m+1,m)=r_{lk}$ and $M_3(m+1, m-1)=r_{kn}$ with $l<n\in \N$.
In the next round \pa\ plays $(0,1,\ldots m-2)$ and re-uses the node $2$ such that $M_4(0,2)=\g_0^{-3}$.
This time we have $M_4(m,m-1)=\r_{jl}$ for some $j<l<n\in \N$.
Continuing in this manner leads to a decreasing 
sequence in $\N$. Now that \pa\ has a \ws\ in $F^{m+3}$, 
by lemma \ref{flat}, $\A\notin \bold S_c{\sf Nr}_m\CA_{m+3}$.

{\bf The non--principal type of co-atoms of $\Tm\At\A$ cannot be omitted in an $m+3$--flat model:} 
Since $\A$ has no complete $m+3$--flat representation, then the algebra   $\C=\Tm\At\A$
has no complete $m+3$--flat representation because $\At\C=\At\A$. Furthermore, $\C$ is countable
since it is generated by the countable set $\At\C$.
Assume that $\C=\Fm_T$ for some countable $L_m$ theory $T$. 
Then using exactly the same argument in the last paragraph of the proof of theorem \ref{OTT} 
replacing `square'; by `flat' we get that the type consisting of co--atoms of $T$, namely, 
$\Gamma=\{\neg \phi: \phi_T\in \At\C\}$ cannot be omitted in an $m+3$--flat model.
\end{proof}
Let $\bold S_d$ be the operation of forming {\it dense subalgebras}.  Then for any class 
$\bold K$ having a Boolean reduct $\bold S_d\bold K\subseteq \bold S_c\bold K$. For Boolean algebras the inclusion is proper.
Let ${\sf CRCA}_m$ denote the class of completely representable $\CA_m$s. The following corollary generalizes the result in\cite{HH}. 
\begin{corollary}\label{c} For any $2<m<n<\omega$, with $n\geq m+3$ the class of 
algebras in $\CA_m$ having  complete $n$--flat representations, and the class ${\sf CRCA}_m$ are not elementary. Furthermore, for any class $\sf K$, such that 
$\bold S_c{\sf Nr}_m\CA_{\omega}\cap {\sf CRCA}_{m}\subseteq {\sf K}\subseteq \bold S_c{\sf Nr}_m\CA_{m+3}$, $\sf K$ is not elementary. 
We  can replace the first $\bold S_c$ by $\bold S_d$. 
\end{corollary}
\begin{proof} 
By lemma \ref{complete} and the previous proof, upon noting that
the two classes $\sf CRCA_m$ and $\bold S_c{\sf Nr}_m\CA_{\omega}$ coincide on atomic algebras having countably many atoms \cite[Theorem 5.3.6]{Sayedneat}, 
we get the required result 
without the last statement. 
For this last statement, we give a sketch of proof. One can define a $k$ rounded game $H_k$, $k\leq \omega$ that is stronger than $G_k$ (in the sense that for any atomic algebra $\C$,
\pe\ has a \ws\ in $H_k(\At\C)\implies$ \pe\ has a \ws in $G_k(\At\C$)), such that if 
\pe\ has a \ws\ in $H_{\omega}(\At\B)$ where $\B$ is a countable atomic $\CA_m$, then $\B$ is not only completely representable (by \pe's \ws\ in $G_{\omega}$ implied by her \ws\ in $H_{\omega}$), 
but  using the stronger part of the game involving other moves, \pe\ can arrange that  $\B$ satsfies that $\At\B\in \At{\sf Nr}_m\CA_{\omega}$ and 
its \d\ completion $\Cm\At\B$ is in ${\sf Nr}_m\CA_{\omega}$. (The last two conditions taken together do not imply that $\B$ itself is in ${\sf Nr}_m\CA_{\omega}$ \cite{conference}).
It can be shown that \pe\ has a \ws\ in $H_k(\At\A)$, for all $k<\omega$, where
$\A$ is the rainbow algebra based on $\Z$ and $\N$ used in the previous proof. Thus using ultrapowers together with an elementary chain argument,
we get that $\A\equiv \B$, with $\B$ having the above three properties. This gives the stronger result that any
$\sf K$ between $\bold S_d{\sf Nr}_m\CA_{\omega}(\subseteq \bold S_c{\sf Nr}_m\CA_{\omega}$) and $\bold S_c{\sf Nr}_m\CA_{m+3}$, 
is not elementary, since $\B$ is dense in its \d\ completion $\Cm\At\B.$
\end{proof}
Fix $2<n<\omega$. It is known that ${\sf Nr}_n\CA_{\omega}\subsetneq \bold S_d{\sf Nr}_n\CA_{\omega}$\cite{conference}. 
We do not know whether we can further remove $\bold S_d$ proving that any class beween ${\sf Nr}_n\CA_{\omega}\cap {\sf CRCA}_n$ and ${\bold S}_c{\sf Nr}_n\CA_{n+3}$ 
is not elementary.
\subsection{Finitizability via guarding and relativized representations}
 
Throughout this subsection, unless otherwise indicated, $n$ is a finite ordinal $>1$.
Here we study 
{\it globally guarded, or simply guarded} 
fragments of $L_n$ (first order logic restricted to the first $n$ variables.)
The following theorem is known \cite{v}. It relates the semantics of a formula $\phi$ 
in a `generalized model' to the semantics of its {\it guarded version}, denoted by $\sf guard(\phi)$,  in the standard part of the model expanded with the guard. 
\begin{theorem} 
Let $L$ be a signature taken in $L_n$. 
Let $(M, V)$ be a generalized model in $L$,
that  is, $M$ is a first order $L$--structure and $V\subseteq {}^nM$ is the set of {\it admissible assignments}. Assume that $R$ is an $n$-ary
relation symbol outside $L$. For $\phi$ in $L$,  let $\sf guard(\phi)$ be the formula obtained from $\phi$ by relativizing all quantifiers
to one and the same atomic formula $R(\bar{x})$ and let 
${\sf Guard}(M,V)$ be the model expanding $M$ to $L\cup \{R\}$ by
interpreting $R$ via $R(s)\iff s\in V$. 
Then the following holds:  
 $$M, V, s\models \phi\Longleftrightarrow {\sf Guard} (M, V),s\models {\sf guard}(\phi),$$
where $s\in V$ and $\phi$ is a formula.
\end{theorem}
We will shortly discover that our finitizabitity result 
is in fact an infinite 
analogue of the polyadic equality analogue of the classical Andr\'eka--Thompson--Resek theorem \cite{AT} proved by 
Ferenczi \cite{Fer}. The algebras studied in the last two references are the modal algebras of two `guarded fragments' of 
$L_n$  where  in the generalized models the  (admissable) assignments are restricted to so--called {\it diagonizable and locally square} subsets 
of $^nM$, respectively, to be defined in a moment.
Let us start with a precise algebraic formulation of the finitizability problem for finite dimensions, due to Maddux,  N\'emeti \cite{M, Nemeti} and others:

{\bf Let $n$ be a finite ordinal $>2$. Can we expand the signature of $\sf RCA_n$
by finitely many permutation invariant operations
so that  the resulting new variety of set algebras, namely, the variety of representable
algebras of dimension $n$ in this new signature, is finitely axiomatizable?}

Here  permutation invariance 
is a necessary condition if we want isomorphic models to satisfy the same formulas,
a basic requirement in abstract model theory.  
Tarski called such operations {\it logical} \cite{Sain}. 
The substitution operations ${\sf s}_{\tau}$ with $\tau\in \T$ ($\T$ a rich semigroup) are permutation invariant. 
The notion of permutation invariance is discussed at length in \cite{Sain, SG, Bulletin} and it tends to keep the problem on the tough side.
But via relativization (without the need to expand the signature) the following theorem can be proved. But first a definition.

\begin{definition}\label{de}  Let $\alpha$ be any ordinal. A set $V\subseteq {}^\alpha U$ {\it diagonalizable} if whenever 
$s\in V$ and $i<j<\alpha$, then $s\circ [i|j]\in V$.
$V$ is {\it locally square} if whenever $s\in V$ and $\tau: n\to n$, then $s\circ \tau\in V$.
\end{definition}
Unions of cartesian spaces and weak cartesian spaces are locally square. 
In particular, disjoint such unions are locally square. 
The part dealing with finite axiomatizability in the next theorem is nothing more than the celebrated Andr\'eka--Resek--Thompson result \cite{AT}
and its polyadic--equality analogue due to Ferenczi \cite{Fer,f}. 
Decidability is proved in \cite{ans}.

Recall that $\mathfrak{B}(V)$ is the Boolean algebra $(\wp(V), \cup, \cap, \sim, \emptyset, V)$.
We denote the the class of set algebras of the form
$(\mathfrak{B}(V), {\sf c}_i, {\sf d}_{ij})_{i, j<n}$ where $V$ is diagonalizable by ${\sf D}_n$ and that consisting of algebras of the form 
$(\mathfrak{B}(V), {\sf c}_i, {\sf d}_{ij}, {\sf s}_{[i,j]})_{i, j<n}$ where $V$ is locally square by ${\sf G}_n$.
\begin{theorem}\label{m}
\cite{AT, f, Fer, ans}.  
Fix  $n>1$. Then  ${\sf D}_n$ and ${\sf G}_n$ are varieties that are axiomatizable by a finite schemata. 
In case $n<\omega$, both varieties are finitely axiomatizable and have a decidable universal (hence equational) theory.
\end{theorem}
\begin{proof} 
In the coming first three items we assume that $1<n<\omega$.

{\bf (1) Decidability:} We give a new (to the best of our knowledge) proof for 
decidability of the universal theory of ${\sf G}_n$. The proof is inspired
by the proof of \cite[Lemma 19.14]{book} which depends on the decidability
of the loosely guarded fragment of first order logic.
 
For $\A\in {\sf G}_n$,   let $\L(\A)$ be the first order signature consisting of an
$n$--ary relation symbol for each element
of $\A$. 
Then we show that for every $\A\in {\sf G}_n$, for any  $\psi(x)$ a quantifier free formula of the signature of $\sf G_n$
and $\bar{a}\in \A$ with $|\bar{a}|=|\bar{x}|$, there is a loosely guarded $\L(\A)$
sentence $\tau_{\A}(\psi(\bar{a}))$ whose relation symbols are among $\bar{a}$
such that for any relativized representation $M$ of $\A$,
$\A\models \psi(\bar{a})\iff M\models \tau_{\A}(\psi(\bar{a}))$. 

Let $\A\in \sf G_n$ and 
$\bar{a}\in \A$. We start by the terms. Then by induction we complete the translation to quantifier free formulas.
For any tuple $\bar{u}$ of distinct $n$ variables, and term $t(\bar{x})$ in the signature
of $\sf G_n$,
we translate $t(\bar{a})$ into a loosely guarded formula
$\tau_\A^{\bar{u}}(t(\bar{a}))$ 
of the first order language having signature
$L(\A)$.
If $t$ is a  variable, then $t(\bar{a})$ is $a$ for some $a\in \rng(\bar{a})$, and we
let $\tau_\A^{\bar{u}}(t(a))=a(\bar{u}).$
For ${\sf d}_{ij}$ one sets
$\tau_{\A}^{\bar{u}}(t)$
to be ${\sf d}_{ij}^{\A}(\bar{u})$ and the constants $0$ and $1$ are handled analogously.
Now assume inductively that $t(\bar{a})$ and $t'(\bar{a})$
are already translated.
We suppress $\bar{a}$ as it plays no role here. For all $i, j<n$ and $\sigma:n\to n$, define
(for the clause ${\sf c}_i$,  $w$ is a new variable):
\begin{align*}
\tau_{\A}^{\bar{u}}(-t)&=1(\bar{u})\land \neg \tau_{\A}^{\bar{u}}(t),\\
\tau_{\A}^{\bar{u}}(t+t')&=\tau_{\A}^{\bar{u}}(t)+\tau_{\A}^{\bar{u}}(t),\\
\tau_{\A}^{\bar{u}}({\sf c}_it)&=1(\bar{u})\land \exists w[(1(\bar{u}^i_w) \land \tau_{\A}^{{\bar{u}^i_w}}(t)],\\
\tau_{\A}^{\bar{u}}({ \s}_{\sigma} t)&= 1(\bar{u}) \land(\tau_{\A}^{\bar{u}\circ\sigma}(t)).
\end{align*}
Let $M$ be a relativized representation of $\A$,
then $\A\models t(\bar{a})=t'(\bar{a})$
$\iff M\models \forall\bar{u} [\tau_{\A}^{\bar{u}}(t(\bar{a})) \longleftrightarrow  \tau_{\A}^{\bar{u}}(t'(\bar{a}))].$
For terms $t(\bar{x})$ and $t'(\bar{x})$ and $\bar{a}\in \A$, choose pairwise distinct variables
$\bar{u}$, that is for $i<j<n$, $u_i\neq u_j$
and define
$\tau_{\A}(t(\bar{a})=t'(\bar{a})): = \forall \bar{u} [1(\bar{u})\to (\tau_{\A}^{\bar{u}}(t(\bar{a}))\longleftrightarrow
\tau_{\A}^{\bar{u}}(t'(\bar{a})))].$
Now extend the definition to the Boolean operations as expected, thereby completing the translation of any quantifier free formula
$\psi(\bar{a})$ in the signature of $\sf G_n$  to the $L(\A)$ formula 
$\tau_{\A}(\psi(\bar{a}))$.

Then it is easy to check that, for any
quantifier free formula $\psi(\bar{x})$ in the signature of $\sf G_n$ 
and $a\in \A$, we have:
$$\A\models \psi(\bar{a})\iff M\models \tau_{\A}(\psi(\bar{a})),$$
and the last is a loosely guarded $\L(\A)$  sentence.
By decidability of the loosely guarded fragment
the required result follows. 

{\bf (2) Representability:} The proof in \cite{f} of representability is a step--by--step argument. We re--prove (differently) representability using games.
 In our proof we use the axiomatization in \cite[Definition 6.2.5]{f} 
where all substitution operations $\s_{\tau}$,  $\tau:n\to n$
are in the signature satisfying
$\s_{\tau\circ \lambda}=\s_{\tau}\s_{\lambda}$ and
$\s_{\tau}{\sf d}_{ij}={\sf d}_{\tau (i)\tau (j)}.$ 

The proof is inspired by the proof of \cite[Lemma 7.8]{book}. Details skipped can all be found in 
\cite{Fer}.  Fix $1<n<\omega$. 
Let $\Sigma$ be given as in \cite[Definition 6.2.5]{f}. 
For $i,j\in n$, $i\neq j$,
define ${\sf t}_j^ix={\sf d}_{ij}\cdot {\sf c}_i x$ and $\t_i^ix=x$.
Then $(\t_j^i)^{\A}: \At\A\to \At\A$ \cite{AT}.
We show that if $\A\models \Sigma$ and $\A$ is atomic, then $\A$ is completely representable as an atomic $\sf G_n$. 

A {\it partial network} is defined like a network except that it is a partial 
map whose domain is locally square, and if $N$ is such a network
then we require that it satisfies 
$\s_{[i, j]}N(x)=N(x\circ [i,j])$ for $x\in \dom(N)$ and $i<j<n$. 
Let $\A\in {\sf TA}_n$. Fix an atom $a\in \At\A$. Let $\bar{x}$ be any $n$-tuple of nodes such that 
$x_i=x_j\iff a\leq {\sf d}_{ij}$ for all $i<j<n$.
Let $\map_{\bar{x}}={}^{n}\{x_0, x_1, \cdots, x_{n-1}\}$.  
Consider the following equivalence relation $\sim$ on $\map_{\bar{x}}$:
$\bar{y}\sim\bar{z}\iff \bar{z}=\bar{y}\circ \tau\text{ for some finite permutation }\tau$,
and $\bar{y},\bar{z}\in \map_{\bar{x}}$.
Choose and fix representative tuples for the equivalence classes of $\sim$
such that each tuple is of the form $\bar{x}\circ [i_k|j_k]\ldots \circ [i_0|j_0]$  for some $k\geq 0$, $\si, \sj< n$.

Let $\sf{Rt}$ denote this fixed set of representative tuples. 
Define the map (network) ${N_{\bar{x}}^{(a)}:\map_{\bar{x}}\rightarrow \At\A}$ as follows:
If $\bar{y}\in\sf{Rt}$, then $\bar{y}$ is non--surjective, so it is a composition of $\bar{x}$ with some 
replacements on $n$. Assume that 
$\bar{y}=\bar{x}\circ [i_k|j_k]\ldots \circ [i_0|j_0]$, say,  for some  $\si$, $\sj$ $<n$. Let $N_{\bar{x}}^{(a)}(\bar{y})=\ttr a$.
The number and order of  replacements are not unique of course but the merry go round identities ($\sf MGR$) implied by $\Sigma$ \cite{Fer}, make $\ttr a$ well-defined. 

In more detail, let $\Omega=\{\t_i^j: i,j\in n\}^*$, where for any set $H$, $H^*$ denotes the free monoid generated by $H$. Let
$\sigma=\t_{j_1}^{i_1}\ldots \t_{j_n}^{i_n}$
be a word. Then define for $a\in A$ and $\sigma\in \Omega$,
$\sigma^{\A}(a)=(\t_{j_1}^{i_1})^{\A}((\t_{j_2}^{i_2})^{\A}\ldots (\t_{j_n}^{i_n})^{\A}(a)\ldots ),$
and $\hat{\sigma}=[i_n|j_n]\circ  [i_{n-1}|j_{n-1}]\ldots\circ [i_1|j_1].$
Then using the $\sf MGR$ one can prove that for all $\sigma, \tau\in \Omega$:
$\A\models \sigma(x)=\tau(x)\iff\hat{\sigma}=\hat{\tau}.$
That is $(\forall \sigma,\tau\in \Omega)(\hat{\sigma}=\hat{\tau}\implies (\forall a\in A)(\sigma^{\A}(a)=\tau^{\A}(a)).$

If $\bar{z}=\bar{y}\circ \sigma$ for some finite permutation $\sigma$ and some $\bar{y}\in\sf{Rt}$, then let 
$N_{\bar{x}}^{(a)}(\bar{z})= \s_{\sigma}N_{\bar{x}}^{(a)}(\bar{y})$.
Then it can be checked that $N_{\bar{x}}^{(a)}$ is well defined such that for any $\tau\in {}^nn$,
$N_{\bar{x}}^{(a)}(\bar{x}\circ \tau)=\s_{\tau}a$. 

Now we show that \pe\ has a \ws\ in the atomic game  
of (possibly transfinite) length $|\At\A|+\omega$ as defined in \cite[Definition 3.3.2]{HHbook2}.
In this game \pa\ is offered only a {\it cylindrifier move} and it suffices to check \pe's response to this move. 
(The rest follows by transfinite induction). 

Suppose that  we are at round $t$ and \pa\ chooses $i<n$, an atom $b\in\At\A$, a
previously played partial network
$N_t$ and $\bar{x}\in {}^nN_t$,
such that $N_t(\bar{x})\leq {\sf c}_ib$. 
If there is $z\in N_t$ with $N_t(\bar{x}_z^i)\leq b$ she lets $N_{t+1}=N_t$. This finishes her move.
Else, she takes  $z\notin \rng({\bar{x}})$, $\bar{t}=\bar{x}_z^i$
and defines the partial network $N_{\bar{t}, b}=N^{(b)}_{\bar{t}}\upharpoonright D$, where $D=\{s\in N_{\bar{t}}^{(b)}: z\in \rng(s)\}$.  
She defines $G=N^{(a)}_{\bar{x}}\cup N_{b, \bar{t}}$ where $a=N_t(\bar{x})$ (is an atom in $\A$). 
Then $\dom (N^{(a)}_{\bar{x}})\cap \dom (N_{b, \bar{t}})=\emptyset$, so
$G$ is a partial (map) network. By construction  we have  $N_t(i_1,\ldots, i_n)=G(i_1,\ldots, i_n)$ for all $i_1,\ldots, i_n\in \rng(\bar{x})$ so 
$N_{t+1}=N_t\cup G$ is a partial network which is the required response, since $N_t\subseteq N_{t+1}$, 
$\bar{t}\equiv_i x$ and $N_{t+1}(\bar{t})=G(\bar{t})=N_{b, \bar{t}}(\bar{t})=N_{\bar{t}}^{(b)}(\bar{t})=b$.

For each $a\in \At\A$, consider the play of the game in which \pe\ plays partial networks with fewer than $|\At\A|+\omega$ nodes, 
and \pa\ chooses the atom $a$ initially, and picks all possible
$i<n$, all hyperedges and all legitimate atoms eventually. Let the limit of the play
be $N_a$; $N_a=\bigcup_{t<|\At\A|+\omega} N_t$ with atomic labels defined the obvious way: If $\bar{x}\in \dom(N_a)$, then there exists $t<|\At\A|+\omega$,
such that $\bar{x}\in \dom(N_t)$. One sets $N_a(\bar{x})=N_t(\bar{x})$. This is well defined because the partial networks are nested. 
Then we can assume that for each $a\in \At\A$  there is $\bar{x}\in \dom(N_a)$ with $N_a(\bar{x})=a$, and whenever $\bar{x}\in \dom(N_a)$, $b\in \At\A$         
and $N_a(\bar{x})\leq {\sf c}_ib$, there is a $\bar{y}\in \dom(N_a)$ with
$\bar{x}\equiv_i \bar{y}$ and $N_a(\bar{y})=b$. By re--naming nodes of networks, one can
arrange that $\nodes(N_a)\cap \nodes(N_b)=\emptyset$ whenever $a$ and
$b$ are distinct atoms. The base of the representation is the union
of sets of nodes of the $N_a$s ($a\in \At\A$), and the atomic, hence
complete representation, is defined via the map
$d\mapsto \{\bar{x}: \exists a\in \At\A: \bar{x}\in \dom(N_a), N_a(\bar{x})\leq  d\}$ $(d\in \A).$

{\bf (3) The ${\sf D}_n$ case:}
Here one takes only the subset $NS_{\bar{x}}$ of 
non--surjective maps in $\map_{\bar{x}}$ with $\bar{x}$ as above.  The atomic labels for the {\it partial network}  
$N_{\bar{x}}^{(a)}$ where $a\in \At\A$ and $\A\in {\sf D}_n$, with domain $NS_{\bar{x}}$ is defined for $\bar{y}\in NS_{\bar{x}}$  
by $N_{\bar{x}}^{(a)}(\bar{y})=\ttr a$, where  $\bar{y}=\bar{x}\circ [i_k|j_k]\ldots \circ [i_0|j_0]$  
for some  $\si$, $\sj$ $<n$ which is well defined by $\sf MGR$. 
Here {\it only replacements} are used, because $\bar{y}$ is not surjective.

{\bf (4) Infinite dimensional case:} Now we show briefly that we can lift the representability result proved above to the transfinite.
This is a known result \cite{Fer,f}. We give a different short proof. 
Let $\alpha\geq \omega$ and let $\A\in \sf TA_{\alpha}$. For any
$n\in \omega$ and injection $\rho:n\rightarrow\alpha$, $\mathfrak{Rd}^{\rho}\A$ as in \cite[Definition 2.6.1]{HMT2} is in ${\sf TA}_{n}$. 
Hence by the representability result for the finite dimensional case proved above, $\mathfrak{Rd}^{\rho}\A\in \bold I{\sf G}_n$ and so it is in $\bold S\mathfrak{Rd}^{\rho}\bold I\sf G_{\alpha}$. Let
$J$ be the set of all finite injective sequences $s$ such that $\rng(s)\subseteq \alpha$. For $\rho\in J$, let
$M_{\rho}=\{\sigma\in J: \rng\rho\subseteq \rng\sigma\}$. Let $U$ be an ultrafilter of $J$ such that $M_{\rho}\in U$ for
every $\rho\in J$. Then for $\rho\in J$, there is a $\B_{\rho}\in \bold I\sf G_{\alpha}$
such that $\mathfrak{Rd}^{\rho}\A\subseteq \mathfrak{Rd}^{\rho}\B_{\rho}$. Let $\mathfrak{C}=\Pi_{\rho\in J}\B_{\rho}/U$; it is in ${\bf UpI}\sf G_{\alpha}=G_{\alpha}$.
Define $f:\A\rightarrow{\bf P}_{\rho\in J}\B_{\rho}$ by $f(a)_{\rho}=a$, and finally 
define the required representing embedding $g:\A\rightarrow\mathfrak{C}$ by $g(a) = f(a)/U$. 
\end{proof}

\section {An overview and summary of results}

To get a grasp of how  difficult the {\it representability problem for $\CA$s} seemed to be in the late sixties of the last century
we quote  Henkin, Monk and Tarski \cite[pp.416]{HMT1}:

`{\it There are two outstanding open problems, one of them is the problem of providing a simple
intrinsic characterization for all representable $\CA$s, the second problem is to find a notion
of representability for
which a general representation theorem could be obtained which
at the same time would be close to geometrical representation in the concrete character and intuitive simplicity.
It is by no means clear that a satisfactory solution
of either of these problem will ever be found or
that a solution is possible'!} (Our exclamation mark).

Later, Henkin, Monk and Tarski formulated the finitizability problem this way:

{\it Devise an algebraic version of predicate
logic in which the class of representable algebras forms a finitely based variety}
\cite{Andreka, Nemeti, Sain,  Fer, SG, Bulletin, AT,  v,  book,  HMT2,  Sagi2, Simon2}.

Since (representable)  $\CA$s were originally designed to algebraize first order logic, the two problems are obviously related. 
Seeing as how the class of representable $\CA$s is a variety, the condition  `finitely based'
in the second quote (which means finitely axiomatizable) is probably the most natural interpretation of the somewhat vague `a simple
intrinsic characterization for all representable $\CA$s' in the first quote, where {\it simple intrinsic characterization} is replaced by the more mathematically 
rigorous {\it simple (finite) equational axiomatization}.

We believe that theorem \ref{final}  reformulated next possibly stands against
Henkin, Monk and Tarski's expectations, for the second problem in the first quote \cite[pp.416]{HMT1} {\it does not} prohibit the option of changing  the semantics, 
that is alter the notion of representability,
as long as it is `concrete and intuitive' enough. This, in turn, possibly indicates that their conjecture as 
formulated in the last two lines of their quote at the 
beginning of this section taken from \cite[pp.416]{HMT1}, was either too hasty or/ and unfounded.

{\bf After all we could find  a variety $\sf Gp_{\T}$  of set algebras, with a natural notion of representability; the operations are interpreted
as {\it concrete set--theoretic operations}
(like Boolean intersection and projections) such that if $\T$ is the rich finitely presented semigroup in \cite{Sain} with finite set $\sf S$ presenting $\T$, then 
$\sf Gp_{\T}$ is  definitionally equivalent to a finitely axiomatizable variety in the signature consisting of the Boolean operations
together with $\{{\sf c}_0, {\sf d}_{01}, {\sf s}_{\tau}: \tau\in \sf S\}$.}

We formulate the next theorem as a 
Stone--like representability result for algebras of 
relations of infinite rank in the form given for $\sf G_n$ in theorem 
\ref{m} to draw the analogy with guarding:
 
\begin{theorem}\label{final} Let $\T$ be a countable rich finitely presented subsemigroup of $(^{\omega}\omega, \circ)$ 
with distinguished elements $\pi$ and $\sigma$. 
Assume that $\sf T$ is presented by the finite set of transformations $\sf S$ such that $\sigma\in \sf S$. 
Then the  class $\sf Gp_T$ of all $\omega$--dimensional set algebras of the form
$(\mathfrak{B}(V), {\sf c}_0, {\sf d}_{01}, {\sf s}_{\tau})_{\tau\in \sf S},$
where $V\subseteq {}^{\omega}U,$  $V$ a non--empty union of cartesian spaces, is a finitely axiomatizable variety.
All the operations ${\sf c}_i, {\sf d}_{ij}$, $i,j\in \omega\sim \{0\}$ are term definable. If $\T$ is strongly rich then all properties in theorem \ref{fp} holds for
$\sf Gp_{\T}$.
\end{theorem} 
Now the logical counterpart of the first part of the previous theorem is:
\begin{theorem}\label{f}
Let $\T$ be a semigroup as specified in the previous theorem. Let $\L_{\T}$ be the algebraizable logic corresponding to ${\sf Gp}_{\T}$ (in the Blok--Pigozzi sense). 
Then the
satisfiability relation $\models_w$ induced by ${\sf Gp}_{\T}$
admits a  finite recursive sound and  complete
proof calculus for the set of type--free valid formula schemata which involves only type--free valid formula schemata  $\vdash$ say,
with respect to $\models_w$, so that $\Gamma\models_w \phi\iff \Gamma\vdash \phi$.  This recursive complete 
axiomatization is a  Hilbert style axiomatization, and there is a translation recursive function $\sf tr$ mapping
$L_{\omega, \omega}$ formulas  to  formulas in $\L_{\T}$
preserving $\models_w$ (but not the usual validity $\models$).
\end{theorem}
Here {\it type--free valid formula schemata} is the plural of {\it type--free valid formula schema}. This is a new notion of validity defined by Henkin et 
al. \cite[Remark 4.3.65, Problem 4.16]{HMT2}, \cite[p. 487]{book}. 

\begin{definition}\label{v} A {\it formula schema} is an element of the set of formulas taken in a signature of $\L_{\T}$.
An {\it instance} of a formula schema
is obtained by substituting formulas for the formula variables, i.e for atomic formulas,
in this formula schema. A formula schema is called {\it type--free valid} if
all of its instances are valid.
\end{definition}
Formulas of the
form $\exists x\exists y\phi\leftrightarrow \exists y \exists x \phi$ that are valid in first order logic may not be valid with 
respect to (the weaker validity relation) $\models_w$, 
so the translation function  ${\sf tr}$ is not `faithful' with respect to Tarskian square semantics. 
In the last item of the next theorem we put some of our new results obtained in theorems \ref{interpolation}
and \ref{fp} against their known weaker analogues formulated  in the first two items. 

\begin{theorem}\label{final2}Let $\T$ and $\sf S$ be as in theorem \ref{final}.
\begin{enumarab}
\item {\bf $\sf FOL$ without equality}: The class $\sf K$ of $\omega$--dimensional set algebras of the form
$(\mathfrak{B}(V), {\sf c}_0, {\sf s}_{\tau})_{\tau\in \sf S}$
where $V$ is a compressed space is a finitely axiomatizable variety \cite{Sain}. 
Furthemore, if $\T$ is strongly rich then $\Fr_{\omega}\sf K$ has the interpolation property \cite{AU} and the class of countable 
completely representable algebras coincides with the class consisting of 
the countable atomic and completely additive algebras \cite{Sayedpa}.

\item {\bf $\sf FOL$ with equality}: The class $\sf K$ of $\omega$--dimensional set algebras of the form 
$(\mathfrak{B}(V), {\sf c}_0, {\sf d}_{01}, {\sf s}_{\tau})_{\tau\in \sf S},$
where $V$ is a disjoint union of cartesian squares, is not a variey, 
but $\sf V= \bold H{\sf K}$ is a finitely 
axiomatizable variety, cf. \cite{SG} and theorem \ref{Sain}. If $\T$ is strongly rich and if $X_1, X_2$
are subsets of the set of free generators of $\Fr_{\omega}\sf V$, $a\in \Sg^{\Rd_{qea}\A}X_1$
and $b\in \Sg^{\Rd_{qea}\A}X_2$ are  
such that $a\leq b$, then there exists $c\in \Sg^{\A}(X_1\cap X_2)$ such that
$a\leq c\leq b$ \cite{typeless}. There are countable atomic algebras, when $\T$ is rich or 
strongly rich, that are not completely representable \cite{conference}.

\item {\bf Solution for $\sf FOL$ with equality in this paper}: The class $\sf K$ of $\omega$--dimensional set algebras 
of the form $(\mathfrak{B}(V), {\sf c}_0, {\sf d}_{01}, {\sf s}_{\tau})_{\tau\in \sf S},$
where $V$ is a  union of cartesian squares, is a finitely 
axiomatizable variety. Furthermore, if $\T$ is strongly rich, then 
$\Fr_{\omega}\sf K$ has the interpolation property and the class of countable completely representable algebras 
coincides with the class of countable atomic algebras.
\end{enumarab}
\end{theorem} 
Modulo altering slightly Tarskian semantics,  not only is (3) substantially stronger than the weaker old solution formulated in (2), 
but it also {\it stronger than the old complete solution in (1) for $\sf FOL$ without equality}.  
The reason is that the condition of {\it complete additivity} is not formulated explicitly 
in the characterization of completely representable countable
algebras. It holds anyway.

\subsection{Summary of results and closely related ones in tabular form} 

In the next table, we summarize our results in tabular form.
We go on to fix the notation. For finite $n$, $L_n$ denotes first order logic with equality restricted to the
first $n$ variables and  $\sf KL$ denotes Keisler's logic \cite{K} with algebraic counterpart $\PA_{\omega}$.   
$\PA_\T$ denotes the reduct of $\PA_{\omega}$ studied by Sain \cite{Sain}, where
$\T$ is a rich finitely presented subsemigroup  of $(^{\omega}\omega, \circ)$ 
and $L_{\T}$ is the corresponding algebraisable (complete) extension of first order logic {\it without equality}. 

We refer to the first seven  rows by table 1. For properties in the upper most row of table 1, $\sf f.a$ is short for finitely axiomatizable, `$\sf CR$ is el' abbreviates, 
that the class of completely representable algebras (in  the class addressed) is elementary. $\sf SUPAP$ 
is short for the super amalgamation property
and  {\sf atom-can} is short hand for 
atom--canonical. 

In the last seven rows of the table which we refer to as table 2, various properties of the logics $L_n$, $\L_G$,  $L_{\omega, \omega}$, $\sf KL$, $L_\T$, and $\L_\T$  
are given, where $\L_G$ is  the algebraizable logic corresponding to 
${\sf G}_n$.  Recall that $\L_{\T}$ is the algebraizable logic corresponding to ${\sf Gp}_{\T}$ with $\T$ a rich semigroup. 

We say that a quantifier logic $\L$ {\it enjoys a Lindstr\"om's theorem, ${\sf LT}$ for short}, 
if $\L$ is  {\it countably compact},  
has {\it L\"owenheim number} \cite[Definition on p.130]{CK} $\omega$,
and $\L$ has the  Craig interpolation property. 
It is well known that for $\L$ extending $L_{\omega, \omega}$ (having the same Tarskian semantics) 
only $\L=L_{\omega, \omega}$ enjoys $\sf LT$ (this is called Lindstr\"om's theorem).  

For properties in the upper most row of table 2, $\sf TF$ is short for `admits a 
finite complete calculus involving only {\it type--free
valid formula schemata} in the sense of definition \ref{v}',  $\sf dec.val$ 
abbreviates that the validity problem is decidable, $\sf OTT$ abbreviates `that an omitting types theorem holds',  $\sf VT$ 
is short for  `Vaught's theorem: Countable atomic theories have atomic models',  $\sf int$ stands for (Craig) interpolation, $\sf alg$ stands for algebraisable,
and finally $\sf LT$ is short for a 
`Lindstr\"om's theorem' as just defined.

In the first column we assume that $\T$ is rich and finitely presented and in all other columns we assume 
that $\sf T$ is strongly rich. In the table $n$ is finite $>2$, and $k\geq 3$ (possibly infinite). 
Without the left hand most column, the results declared in the first four columns in table 2 are the logical counterpart of the results in 
the first four columns in table 1 (using fairly standard `bridge theorems' in 
algebraic logic \cite{bp}). 
We view $\sf RCA_{\omega}$ as the algebraic counterpart of the {\it type--free formalism} of $L_{\omega, \omega}$ given in 
\cite[Section 4.3.28, item(ii)]{HMT2}.

The positive answers for $\sf LT$ for finite variable logics, that do not extend $L_{\omega, \omega}$, 
follow by convention, that is from {\it how we defined $\sf LT$.} 
$\sf KL$ does not have $\sf LT$ because its  L\"owenheim number is not $\omega$ 
since its signature is uncountable. $\sf SUPAP$ and $\sf int$ for $\PA_{\omega}$ and 
$\sf KL$, respectively, are 
proved by the author \cite{super}. 

Sources for other results in the table will be specified right after the table.
\vskip3mm

\begin{tabular}{|l|c|c|c|c|c|c|}    \hline
				{\bf Varieties}		&\sf f.a  &  $\sf SUPAP$    &${\sf decidable}$                     &$\sf CR$ is el.           &{\sf Canonical}                         & {\sf atom-can}  \\

                                                               \hline
                                                                          $\bold S{\sf Nr}_n\CA_{n+k}$&no&no&no&no&yes &no\\

                                                                           \hline
                                                                          $\sf G_n$ &yes&yes&yes&yes&yes&yes\\

                                                                           \hline
                                                                          $\sf RCA_{\omega}$&no&no&no&no&yes&?\\

                                                               \hline
                                                                          ${\sf PA}_{\omega}$&no&yes&no &yes&yes&yes \\

                                                               \hline
                                                                          ${\sf PA}_{\T}$&yes&yes&?&yes&yes&yes \\

                                                               \hline
                                                                          ${\sf Gp}_{\T}$&yes&yes&?&no&yes&yes \\
\hline
\hline				{\bf Logic}		&\sf TF&  $\sf int.$    &\sf dec.val                     & $\sf VT, \sf OTT$            &\sf alg                          & {\sf LT}  \\

                                                               \hline
                                                                          $L_{n}$&no&no&no& no&yes&no \\

                                                                           \hline
                                                                          $\L_G$ &yes &yes& yes&yes&yes&yes\\

                                                                           \hline
                                                                          $L_{\omega, \omega}$ &no &yes& no&yes&no &yes\\

                                                                           \hline
                                                                          $\sf KL$ &no&yes& no&yes&yes &no\\

                                                                           \hline
                                                                          $L_{\T}$ &yes &yes& ?&yes&yes &yes\\

                                                                           \hline
                                                                          $\L_{\T}$ &yes &yes& ?&yes&yes &yes\\

\hline

\end{tabular}

\vskip3mm
We cite the sources for other results in the table and make a few more comments. We count  the rows and columns 
without the upper most row and left hand most column:

\begin{enumarab}

\item The results in the first and fourth  row of table 1, when $k=\omega$, are known classical results for $\CA$s \cite{HMT2, Hodkinson, HH}. For $3\leq k<\omega$
in the first row. 
The `no'  in columns 4 and 6  of the first row in  table  1 is proved in corollaries \ref{hod} and \ref{c}  
refining and strengthening the results in \cite{Hodkinson, HH}. 
The results in the second and third rows of table 1 are mostly summarized in  theorem \ref{m}, see also 
\cite{AT, Fer, f}.  The $\sf SUPAP$ is proved for ${\sf G}_n$ and ${\sf D}_n$ in \cite{marx}.
The rough idea, using the terminology and notation, in {\it op.cit} is the following. We know that $\sf G_n$ is axiomatized by a set of positive equations, 
so is canonical.
The first order correspondents of this set of positive equations translated to the class
$\bold L={\sf Str}({\sf G}_n)$ will be {\it Horn formulas},
hence {\it clausifiable} and so $\bold L$ is closed under finite {\it zigzag products}. By
\cite[Lemma 5.2.6, pp.107]{marx}, ${\sf G}_n$ has the super amalgamation property. 
Worthy of note is that this technique works verbatim for $\sf Gp_T$.  

\item The positive results in the seventh  row of table 1 for ${\sf Gp}_{\T}$ are the essential results in this paper, for the infinite dimensional case, 
proved in theorems \ref{interpolation}, \ref{fp}, \ref{final} and \ref{final2}. The only `no' in this line, namely, that 
the class of completely representable algebras is {\it not} elementary, is proved in the second item of 
theorem \ref{interpolation}.

\item   It is known that for all undecidable logics addressed in the table, the validity problem is 
recursively enumerable, except for $\sf KL$.  For first order logic the validity problem  is undecidable. 
In $\L_{\sf T}$ where $\T$ is rich and finitely presented,
the intuitive implication {\it completeness $\implies$ recursive enumerability of validities} holds and 
it is likely that the 
equational theory of $\bold I{\sf Gp}_{\T}$ is decidable, hence so is the validity problem of $\L_{\sf T}$. The finitizability problem as 
posed by Henkin, Monk and Tarski  does not require decidability of the validity problem for the corresponding algebraisable logic. 
 
\item  If $\T$ is a rich semigroup, then 
a set algebra in  ${\sf Gp}_\T$ has top element a union of cartesian squares, while a set algebra in ${\sf G}_{\omega}$ has  top element a union of weak 
spaces. Both $\sf G_{\omega}$ and $\sf Gp_{\T}$ 
are axiomatizable by a finite schemata, 
but $\sf Gp_{\T}$ has the advantage that it can be finitely axiomatized if $\T$ happens to be finitely presented.  In both cases the decidability of their equational theory remains unsettled.

\item For $\sf CPEA_{\alpha}$, $\alpha$ an infinite ordinal, the following hold: {\sf SUPAP, atom-can, CR. el.} and {\sf canonicity}. 
In fact, we have that the class of completely representable algebras 
coincides with the class of atomic ones with {\it no restriction on cardinalities} which was the case with $\sf CPEA_{\T}$, as proved in theorem \ref{main}. 
{\sf SUPAP} can be proved by either the technique sketched in the first item, or exactly like the proof in the last item of theorem
\ref{interpolation} by proving interpolation for the free algebras. This is done by dilating the given free algebra $\A$ with $\beta$--generators, $\beta$ a non--zero cardinal, 
to a regular cardinal $\mathfrak{n}>max \{|A|, |\alpha|, \beta\}$; the rest of the proof is the same. 
\end{enumarab}
Though admitting a finite schemata axiomatizability, the variety $\PA_{\omega}$ 
has a lot of drawbacks from the recursion
theoretic viewpoint.
In Keisler's logic  though the set of validities can be captured by a finite schemata,
namely, Halmos' schemata, this set is not recursively enumerable 
\cite{Sain} which is not the case with $\L_{\T}$ (the logic corresponding to $\sf Gp_{\T}$) when $\T$ is (only) rich. 
The same can be said about ${\sf CPEA}_{\omega}$ due to the presence of continuum many substitution operators in 
its signature. A good reference for excluding apparently satisfactory solutions to the finitizability problem 
(like $\PA_{\omega})$ is \cite{Simon2} entitled: {\it What the finitization problem is not?}
In this paper we focus more on {\it what it is.}

{\bf In the second part of theorem \ref{m} concerning $\sf G_n$ and in theorem \ref{final},
commutativity of cylindrifiers is (syntactically) weakened and  semantics are accordingly
relativized to unions of  spaces that are not necessarily disjoint
to obtain a finitely axiomatizable variety of representable algebras
corresponding in the Blok--Pigozzi sense to the algebraizable formalisms of
quantifier first logic with equality  having $2<n\leq \omega$ variables.
What is highly significant is that in both cases the relativization {\it is the same}.}

There are weaker versions of the finitizability problem ($\sf FP$), like seeking only a `finite recursive schemata', or asking that the 
class of set algebras $(\sf Set)$ {\it generates} a finitely axiomatizable variety like in theorem \ref{Sain}. The class $\sf Set$ itself may not be a variety, 
{\it not even a quasi--variety}. Worthy of note, is that Tarski \cite{Sain} formulated the $\sf FP$ for relation algebras in the last form.

In this paper we provided a solution (in $\sf ZFC$) to the most strict version of the $\sf FP$ for $L_{\omega, \omega}$ 
posed by Henkin, Monk and Tarski in the seventies of the last century 
for $L_{\omega, \omega}$ modulo  (what we believe to be) a reasonable relativization or guarding of semantics.
The relativization is not so severe; item (9) in definition \ref{capa} roughly says that 
substitutions and cylindrifiers commute one way.

Research in 
algebraic logic over the last three decades has revealed that full fledged commutativity of cylindrifiers  is `the main culprit' responsible for many negative results. 
In essence, a precarious `Church--Rosser' condition, it is  responsible for robust undecidability and non--finite axiomatizability when the dimension is at least three. 
We have seen in theorems \ref{2.12}, \ref{decidable}, \ref{OTT} and \ref{OTT2}, 
that the analogous negative results in the classical case addressing the class ${\sf RCA}_n$ ($n\geq 3)$ proved in \cite{Andreka}, \cite[Theorem 4.2.18]{HMT2}, \cite{Hodkinson, HH}, respectively, 
are not avoided even if cylindrifiers are allowed to commute {\it only locally}. By lemma \ref{neatsq} (relating $k$--flatness to existense of $k$--dilations)
this amounts to working with the larger varieties $\bold S{\sf Nr}_n\CA_k$ ($n+3\leq k<\omega$). For such proper approximations of $\sf RCA_n$ negative properties persist. 
Here by `proper approximations' we mean that for $2<n<\omega$ and positive $k\geq 3$, 
${\sf RCA}_n\subsetneq \bold S{\sf Nr}_n\CA_{n+k}$  and 
$\bigcap_{k\in \omega, k\geq 3} \bold S{\sf Nr}_n\CA_{n+k}={\sf RCA}_n.$
But when we weakened commutativity of cylindrifiers {\it globally}, we succeeded to 
obtain positive results, using the {\it same relativization} for both finite and infinite dimensions, 
formulated and proved in theorems \ref{fp}, \ref{m}, \ref{final}, \ref{f}, \ref{final2}. 

To the best of our knowledge {\it no solution exists}  to the $\sf FP$ requiring only finite axiomatizability (completeness and recursive enummerability of validities)
when we require that top elements of set algebras are a {\it disjoint union} of cartesian spaces, 
unless the ontology, namely, the underlying set theory is 
changed \cite{Nemeti, Bulletin, Sayedneat, Simon2}.
This is done by weakening the axiom of foundation \cite[p.130]{Sayedneat}.  
Our investigation in this paper is by no means final. We summarize the above discussion in the following queries:

\begin{enumarab}

\item Is there a solution in $\sf ZFC$ to the $\sf FP$ if we require that the top elements of representable algebras are  {\it disjoint} unions of cartesian squares?
Is removing the condition of {\it disjointness} necessary or only sufficient?

\item Is there a countable semigroup on $\omega$ that is {\bf both} {\it finitely presented and strongly rich}? 
\item Given a countable finitely presented semigroup $\T$ on $\omega$, 
is the validity problem of $\L_\T$ decidable?

\end{enumarab}

\end{document}